\documentclass[12pt]{amsart}

\usepackage{amsmath}
\usepackage{amssymb}
\usepackage{amstext}
\usepackage{amscd}
\usepackage{amsthm}

\usepackage[margin=3cm]{geometry}

\usepackage{color}

\usepackage[colorlinks]{hyperref}
\hypersetup{linkcolor=blue, citecolor=red}

\newtheorem{teo}{Theorem}[section]
\newtheorem{prop}[teo]{Proposition}
\newtheorem{lem}[teo]{Lemma}
\newtheorem{cor}[teo]{Corollary}

\newtheorem{defini}[teo]{Definition}

\newtheorem{rem}[teo]{Remark}

\newcommand{\der}{{\rm der}}

\newcommand{\ad}{{\rm ad}}

\newcommand{\RR}{{\mathbb R}}
\newcommand{\ZZ}{{\mathbb Z}}
\newcommand{\QQ}{{\mathbb Q}}
\newcommand{\NN}{{\mathbb N}}

\newcommand{\GG}{{\mathbb G}}

\newcommand{\cH}{{\mathcal H}}
\newcommand{\cE}{{\mathcal E}}

\newcommand{\cA}{{\mathcal A}}
\newcommand{\cQ}{{\mathcal Q}}
\newcommand{\cP}{{\mathcal P}}

\newcommand{\gG}{{\bf G}}
\newcommand{\gH}{{\bf H}}

\newcommand{\gN}{{\bf N}}
\newcommand{\gA}{{\bf A}}
\newcommand{\gM}{{\bf M}}

\newcommand{\gSL}{{\bf SL}}

\newcommand{\gP}{{\bf P}}
\newcommand{\gZ}{{\bf Z}}
\newcommand{\gQ}{{\bf Q}}
\newcommand{\gS}{{\bf S}}

\title[The space of homogeneous probability measures on $\overline{\Gamma\backslash X}^S_{\rm max}$]{The space of homogeneous probability measures on $\overline{\Gamma\backslash X}^S_{\rm max}$ is compact}

\author[Daw\and Gorodnik\and Ullmo]{Christopher Daw\and Alexander Gorodnik\and Emmanuel Ullmo}

\address{Daw: Department of Mathematics and Statistics, University of Reading,
    White\-knights,  PO Box 217,  Reading,  Berkshire RG6 6AH,  United Kingdom}
\email{chris.daw@reading.ac.uk}

\address{Gorodnik: Institut f\"ur Mathematik, Universit\"at Z\"urich, Winterthurerstrasse 190, 8057 Z\"urich, Switzerland}
\email{alexander.gorodnik@math.uzh.ch}

\address{Li:
 Institut f\"ur Mathematik, Universit\"at Z\"urich, Winterthurerstrasse 190, 8057 Z\"urich, Switzerland}
\email{jialun.li@math.uzh.ch}

\address{Ullmo:
  IHES, Universit\'e Paris-Saclay, Laboratoire CNRS Alexander Grothendieck, Le Bois-Marie 35, route de Chartres, 91440 Bures-sur-Yvette, France}

\email{ullmo@ihes.fr}



\begin{document}
\maketitle
\begin{center}
\textsc{With an appendix by Jialun Li} 
\end{center}

\begin{abstract} 
In this paper we prove that the space of homogeneous probability measures on the maximal Satake compactification of an arithmetic locally symmetric space is compact. As an application, we explain some consequences for the distribution of weakly special subvarieties of Shimura varieties.
\end{abstract}

\tableofcontents

\section{Introduction}

In this paper, we study the behaviour of sequences of homogeneous measures. More specifically, given a sequence of such measures, we will be interested in describing its limit points. This problem has been studied by Eskin, Mozes, and Shah \cite{MS:ergodic}, \cite{EMS:flows}, \cite{EMS:nondivergence}, 
who showed that, under certain conditions, any limit point is either a homogeneous measure itself or a zero measure. The later case amounts to the existence of a subsequence of measures diverging to infinity and, in \cite{EMS:nondivergence}, a non-divergence condition was established. Such results concerning the convergence of measures have found several remarkable applications in arithmetic geometry (see, for instance, \cite{EMS:flows}, \cite{CU:strong}, \cite{GO11}).
However, the applicability of these tools have so far been limited to the case in which divergence to infinity can be ruled out. 
The goal of the present paper is to investigate limits of divergent sequences by considering them 
inside a Satake compactification. Ultimately, we show that any limit point
is also a homogeneous measure supported on precisely one of the boundary components of the compactification.

We conjectured this result in our previous paper \cite{DGU}, wherein we developed several tools with which to study it and also proved some particular cases, including the locally symmetric space associated with $\gSL_3(\RR)$. We refer to the introduction of \cite{DGU} for some further historical background in homogeneous dynamics. We simply recall here the importance, for our purposes, of the seminal works of Ratner \cite{ratner:measure}, \cite{ratner:flows} on the dynamics of unipotent flows, and of some of its developments by Dani--Margulis \cite{DM:uniflows} and Eskin--Mozes--Shah, as alluded to above. 

\subsection*{Formulating the main result}

Let $\bf{G}$ be a semisimple algebraic group defined over $\QQ$ and let $G$ denote the connected component of $\gG(\RR)$ containing the identity. Let $K$ be a maximal compact subgroup of $G$ and let $\Gamma\subset \gG(\QQ)\cap G$ be an arithmetic lattice. Denote by $X$ the associated Riemannian symmetric space $G/K$, denote by $x_0$ the point in $X$ with stabilizer equal to $K$, and denote by $S$ the associated arithmetic locally symmetric space $\Gamma\backslash X$. Let $\cP(S)$ denote the set of Borel probability measures on $S$.

A $\QQ$--algebraic subgroup ${\bf H}\subset {\bf G}$ is said to be of type $\cH$ if 
the radical $\bf R_H$ of ${\bf H}$ is unipotent and the real Lie groups underlying the $\QQ$--simple factors of ${\bf H}$ are not compact. With an algebraic subgroup ${\bf H}\subset{\bf G}$ of type $\cH$ and some $g\in G$ we can associate a probability measure $\mu_{{\bf H},g}\in\cP(S)$ with support equal to $\Gamma\backslash\Gamma Hgx_0\subset S$.
Such a measure is called homogeneous and we denote by 
 $$
 \cQ(S):=\{\mu_{{\bf H}, g},\ {\bf H} \mbox{ of type $\cH$},\ g\in G\}\subset \cP(S)
 $$
the set of homogeneous probability measures on $S$.

The maximal Satake compactification of $S$
has a decomposition
 \begin{equation}
 \overline{\Gamma\backslash X}_{\rm max}^S= \Gamma\backslash X\coprod \coprod_{{\bf P}\in \cE }\Gamma_{ X_P}\backslash X_P
 \end{equation}
where ${\bf P}$ varies among a (finite) set of representatives $\cE$ of the $\Gamma$--conjugacy classes of proper $\QQ$--parabolic subgroups of ${\bf G}$ and the boundary component $\Gamma_{X_P}\backslash X_P$ is the arithmetic locally symmetric space associated with $\gP$. As a consequence, for any boundary component  $\Gamma_{X_P}\backslash X_P$ of $ \overline{\Gamma\backslash X}_{\rm max}^S$, we can define the set 
$\cQ(\Gamma_{X_P}\backslash X_P)$ of homogeneous probability measures on  $\Gamma_{X_P}\backslash X_P$ as we defined $\cQ(S)$ for the open boundary component $S=\Gamma\backslash X$ of $\overline{\Gamma\backslash X}_{\rm max}^S$.  A probability measure $\mu$ on  
$ \overline{\Gamma\backslash X}_{\rm max}^S$ is said to be homogeneous if $\mu$ is homogeneous on $S$  or on one 
of the proper boundary components $\Gamma_{X_P}\backslash X_{P}$. Then 
$$
\cQ( \overline{\Gamma\backslash X}_{\rm max}^S)=\cQ(S) \coprod \coprod_{{\bf P}\in \cE }\cQ(\Gamma_{X_P}\backslash X_P)\subset \cP( \overline{\Gamma\backslash X}_{\rm max}^S)
$$
is the set of homogeneous probability measures on $\overline{\Gamma\backslash X}_{\rm max}^S$. Our main result is the following, which establishes \cite[Conjecture 1.1]{DGU}.

\begin{teo}\label{tI1}
\
\begin{enumerate}
\item[(i)] The set $\cQ( \overline{\Gamma\backslash X}_{max}^S)$ of homogeneous probability measures on $\overline{\Gamma\backslash X}_{max}^S$ is compact.

\item[(ii)] Let $({\bf H}_n)_{n\in \NN}$ be a sequence of algebraic subgroups of ${\bf G}$ of type $\cH$ and let $(g_n)_{n\in \NN}$ be a sequence of elements of $G$. Let 
$\mu\in \cP(\overline{\Gamma\backslash X}_{\rm max}^S)$ be a weak limit of the associated sequence $(\mu_{{\bf H}_n,g_n})_{n\in \NN}$ of homogeneous measures on $S=\Gamma\backslash X$. Then $\mu$ is a homogeneous measure on $\overline{\Gamma\backslash X}_{\rm max}^S$.
\end{enumerate}
In (ii), if $\mu_{{\bf H}_n,g_n}\rightarrow\mu$ and $\mu$ is supported on the boundary component $\Gamma_{X_P}\backslash X_P$, then there exists a connected algebraic subgroup $\bf H$ of ${\bf P}$ of type $\mathcal{H}$ and an element $g\in P$ such that $\mu=\mu_{\gH,g}$ and ${\bf H}_n$ is contained in $\bf H$ for $n$ large enough.
\end{teo}

Note that, in Theorem \ref{tI1}, (ii) is an immediate consequence of (i), but (ii) (for any $\gG$) implies (i) by simple properties of Satake compactifications (see Proposition \ref{2implies1} (ii)).

\subsection*{Sequences of weakly special subvarieties}
If we assume, moreover, that $X$ is hermitian, then, by a fundamental result of Baily--Borel \cite{BB:compactification}, the hermitian locally symmetric space $S=\Gamma\backslash X$ has the structure of a quasi-projective algebraic variety. Such varieties have been studied extensively by Shimura and Deligne as a generalization of the modular curve. As such they are usually called Shimura varieties and they now play a central role in the theory of automorphic forms (in particular, the Langlands program), the study of Galois representations, and Diophantine geometry. The main examples of Shimura varieties are given by the moduli spaces $\cA_g$ of principally polarized abelian varieties of dimension $g$, in which case ${\bf G}={\rm Sp}_{2g}$ and $\Gamma= {\rm Sp}_{2g}(\ZZ)$. In general, Shimura varieties are moduli spaces of Hodge structures of a restricted type. They are endowed with special points, special subvarieties, and weakly special subvarieties that play a central role in their theory and are the central objects in the Andr\'e--Oort and Zilber--Pink conjectures \cite{pink:combination}, \cite{pink:generalisation}, \cite{zannier:UI}.
 
Special points of $S$ parametrize `maximally symmetric' Hodge structures (more precisely, Hodge structures whose Mumford--Tate groups are tori). In the case of $\cA_g$, they correspond to abelian varieties with complex multiplication. The weakly special subvarieties of $S$ are the totally geodesic subvarieties, and a special subvariety is a weakly special subvariety containing a special point. A special subvariety can also be described in Hodge theoretic terms as a certain locus of `non-generic' Hodge structures. The relevance of these notions for our purposes is due to the fact that any weakly special subvariety of $S$ is the support of a homogeneous measure. 
 
The equidistribution of sequences of homogeneous measures associated with special subvarieties of Shimura varieties (in situations where there is no escape of mass) has been studied by Clozel and the third author \cite{CU:strong}, \cite{ullmo:NF}, and played a central role in the proof of the Andr\'e--Oort conjecture under the Generalized Riemann Hypothesis \cite{ky:andre-oort}, \cite{uy:andre-oort}.  The very successful strategy of Pila--Zannier \cite{pz:manin-mumford}, \cite{pila:Cn}, which has yielded unconditional cases of the Andr\'e--Oort and Zilber-Pink conjectures, has highlighted the importance of understanding the distribution of weakly special varieties. The Ax--Lindemann conjecture \cite{uy:ax},  \cite{pt:ax}, \cite{kuy:ax}, at the heart of their strategy,  asserts that the Zariski closure of an algebraic flow is weakly special. 
 
The main result of this paper has implications on the equidistribution properties of sequences of weakly special subvarieties, even in the situation when there is escape of mass.  In the Shimura case, the Baily--Borel compactification  $\overline{S}^{BB}$ of $S$  has  the form
$$
\overline{S}^{BB}=S\coprod  \coprod_{{\bf P}\in \cE_{\rm max} } \Gamma_{X_{h,P}}\backslash X_{h,{P}}
$$
where $\cE_{\rm max}$ is 
a set of representatives for the $\Gamma$--conjugacy classes of maximal $\QQ$--parabolic
 subgroups of ${\bf G}$ and each boundary component $\Gamma_{X_{h,P}}\backslash X_{h,{P}}$ is hermitian locally symmetric. As before, we say that a measure $\mu$ on $\overline{S}^{BB}$ is homogeneous if $\mu$ is supported on the open boundary component $S$ or on one of the proper boundary components and is homogeneous. In this situation, we have the following theorem (which is a consequence of Theorem \ref{tI1} by \cite{DGU}, Theorem 3.4).
 
 \begin{teo}\label{tI2}
Let $(Z_n)_{n\in \NN}$ be a sequence of weakly special subvarieties of $S$. Let $(\mu_n)_{n\in \NN}$ be the associated sequence of homogeneous measures. Then in the space $\cP(\overline{S}^{BB})$ of probability measures on $\overline{S}^{BB}$ any weak limit $\mu$ of $(\mu_n)_{n\in \NN}$ is homogeneous.
\end{teo}

One can show that the weak limits in Theorem \ref{tI2} need not be supported on a weakly special subvariety.

\subsection*{Applications and developments}
The main construction of the paper, Theorem \ref{DGUtype} (which, via Theorem \ref{criterion2}, yields Theorem \ref{tI1}), has recently been applied in a new proof (due to Richard and the third author) of the so-called g\'eom\'etrique Andr\'e--Oort conjecture (see \cite{RU}, Th\'eor\`eme 1.3). The proof passes through the following ``dynamic alter ego'' of the aforementioned conjecture.

\begin{teo}[\cite{RU}, Th\'eor\`eme 1.6]\label{RUteo}
Let $V$ be an irreducible algebraic subvariety of $S$ containing a sequence $(Z_n)_{n\in \NN}$ of weakly special subvarieties $Z_n=\Gamma\backslash \Gamma H_nx_n$ for some semisimple $\QQ$--algebraic subgroups $\gH_n$ of $\gG$. There exists a $\QQ$--algebraic subgroup $\gH_\infty$ of $\gG$ such that, after possibly replacing $(Z_n)_{n\in \NN}$ with a subsequence, $\gH_n$ is contained in $\gH_\infty$ and $V$ contains the spaces $\Gamma\backslash \Gamma H_\infty x_n$ for all $n\in\NN$.
 \end{teo} 
 
In fact, the authors obtain the natural generalization of Theorem \ref{RUteo} in the setting for which $S$ is replaced by a general arithmetic quotient and $V$ is a real analytic subvariety definable in an o-minimal structure (see \cite{RU}, Th\'eor\`eme 1.8). This allows the authors, in an appendix with Chen, to complete the work of the latter author \cite{Chen} on the geometric Andr\'e--Oort conjecture for variations of $\ZZ$--Hodge structures (a conjecture due to Klingler \cite{Klingler}).

\bigskip

In a recent preprint, Zhang has announced an extension of Theorem \ref{tI1} in the Borel--Serre compactification, in which the type $\mathcal{H}$ assumption on the $\gH_n$, for $\gH_n=\gH$ fixed, is relaxed (see \cite{Zhang}, Theorem 1.2). 

\subsection*{Organisation of the paper}
In Section \ref{prelim}, we give the necessary preliminaries to clarify the statement of the main result. In particular, we explain why Theorem \ref{tI1} (ii) implies Theorem \ref{tI1} (i). In Section \ref{fprelim}, we give some further definitions and prove some useful results on parabolic subgroups. In Section \ref{tools}, we state the three main tools used in the proof of Theorem \ref{tI1}, namely, two criteria for convergence proved in \cite{DGU}, and an inequality between simple roots and dual weights due to Li. In Section \ref{proof}, we explain an algorithm combining these three tools, which proves Theorem \ref{tI1} and, in the process, determines the boundary component supporting the limit of a convergent sequence of homogeneous measures. The appendix, due to Li, contains a proof of the aforementioned result regarding root systems.

\section*{Acknowledgements}

The authors would like to thank Jialun Li for providing them with a proof of Theorem \ref{rootconj}.

The first and second authors would like to thank the Institut des Hautes Etudes Scientifiques for hosting them several times during the origination and completion of this work. The key breakthrough was made during a particularly fruitful visit in February 2019.

The first author would like to thank the EPSRC for its support via a New Investigator Award (EP/S029613/1). 

The second author was supported by an SNF grant (200021--182089).

\section{Preliminaries}\label{prelim}

In this section, we collect the definitions required in order to explain the main results. We repeat several definitions established in \cite{DGU}.

\subsection{Borel probability measures}\label{bpm}
Let $S$ be a metrizable topological space and let $\Sigma$ be its Borel $\sigma$--algebra. By a {\it Borel probability measure on $S$}, we mean a probability measure on $\Sigma$. We let $\mathcal{P}(S)$ denote the space of all Borel probability measures on $S$. We say that a sequence $(\mu_n)_{n\in\NN}$ in $\mathcal{P}(S)$ {\it converges (weakly)} to $\mu\in\mathcal{P}(S)$ if we have
\begin{align*}
\int_{S}f\ d\mu_n\rightarrow\int_{S}f\ d\mu,\text{ as }n\rightarrow\infty,
\end{align*}
for all bounded continuous functions $f$ on $S$.

\subsection{Algebraic groups}\label{alggrps}
By an {\it algebraic group} $\bf G$, we refer to a linear algebraic group defined over $\QQ$ and by an {\it algebraic subgroup of $\bf G$} we again refer to an algebraic subgroup defined over $\QQ$. We will use boldface letters to denote algebraic groups (which, again, are always defined over $\QQ$). By convention, {\it semisimple} and {\it reductive} algebraic groups are connected.

If $\bf G$ is an algebraic group, we will denote its radical by $\bf R_G$ and its unipotent radical by ${\bf N}_{\bf G}$. We will write ${\bf G}^\circ$ for the (Zariski) connected component of $\bf G$ containing the identity. We will denote the Lie algebra of $\bf G$ by the corresponding mathfrak letter $\mathfrak{g}$, and we will denote the (topological) connected component of ${\bf G}(\RR)$ containing the identity by the corresponding Roman letter $G$. We will retain any subscripts or superscripts in these notations. 

If $\bf M$ and $\bf A$ are algebraic subgroups of $\bf G$, we will write ${\bf Z_M}({\bf A})$ for the centralizer of $\bf A$ in $\bf M$ and ${\bf N_M(A)}$ for the normalizer of $\bf A$ in $\bf M$. We will denote by $G_\QQ$ the intersection ${\bf G}(\QQ)\cap G$ and we will refer to an arithmetic group of $\gG(\QQ)$ contained in $G_\QQ$ as an arithmetic subgroup of $G_\QQ$. 

\subsection{Groups of type $\mathcal{H}$}\label{typeH}
We say that an algebraic group $\bf G$ is {\it of type $\mathcal{H}$} if ${\bf R}_{\bf G}$ is unipotent and the quotient of $\bf G$ by ${\bf R}_{\bf G}$ is an almost direct product of almost ($\QQ$--simple) algebraic groups whose underlying real Lie groups are non-compact. In particular, an algebraic group of type $\mathcal{H}$ has no rational characters. Note that the image under a morphism of algebraic groups of any group of type $\mathcal{H}$ is a group of type $\mathcal{H}$.

\subsection{Homogeneous probability measures on $\Gamma\backslash G$}\label{measures}
Let $\bf G$ denote an algebraic group and let $\Gamma$ denote an arithmetic subgroup of $G_\QQ$. If $\bf H$ is a connected algebraic subgroup of $\bf G$ possessing no rational characters, then there is a unique Haar measure on $H$ whose pushforward $\mu$ to $\Gamma\backslash G$ is a Borel probability measure on $\Gamma\backslash G$. For $g\in G$, we refer to the pushforward of $\mu$ under the right multiplication--by--$g$ map as the {\it homogeneous probability measure on $\Gamma\backslash G$ associated with ${\bf H}$ and $g$}. More explicitly, this is the $g^{-1}Hg$--invariant probability measure supported on $\Gamma Hg$.

\subsection{Parabolic subgroups}\label{parsubgrps}
 A {\it parabolic subgroup} $\bf P$ of a connected algebraic group $\bf G$ is an algebraic subgroup such that the quotient of ${\bf G}$ by $\bf P$ is a projective algebraic variety. In particular, $\bf G$ is a parabolic subgroup of itself. However, by a {\it maximal} parabolic subgroup, we refer to a maximal {\it proper} parabolic subgroup. Note that $\bf R_G$ is contained in every parabolic subgroup of $\bf G$.
 
\subsection{Cartan involutions}\label{cartinv}
Let $\bf G$ be a reductive algebraic group and let $K$ be a maximal compact subgroup of $G$. Then there exists a unique involution $\theta$ on $G$ such that $K$ is the fixed point set of $\theta$. We refer to $\theta$ as the {\it Cartan involution of $G$ associated with $K$}. 
 
\subsection{Boundary symmetric spaces}\label{s2.5}
Let $\bf G$ be a semisimple algebraic group and let $K$ be a maximal compact subgroup of $G$. Let $\bf P$ be a parabolic subgroup of $\bf G$. As in \cite{BJ:compactifications}, (I.1.10), we have the {\it real Langlands decomposition (with respect to $K$)}
\begin{align*}
P=N_PM_PA_P,
\end{align*}
where $L_P:=M_PA_P$ is the unique Levi subgroup of $P$ such that $K_P:=L_P\cap K=M_P\cap K$ is a maximal compact subgroup of $L_P$, and $A_P$ is the maximal split torus in the centre of $L_P$. We denote by $X_P$ the {\it boundary symmetric space} $M_P/K_P$, on which $P$ acts through its projection on to $M_P$.

\subsection{Maximal Satake compactifications}\label{sat}

Let $\gG$ be a semisimple algebraic group and let $K$ be a maximal compact subgroup of $G$. Denote by $X$ the symmetric space $G/K$ and let $\Gamma$ denote an arithmetic subgroup of $G_\QQ$. We let
\begin{align*}
_{\QQ}\overline{X}^S_{\rm max}:=\coprod_{\bf P}X_P,
\end{align*}  
where $\bf P$ varies over the (rational) parabolic subgroups of $\bf G$. We endow $_{\QQ}\overline{X}^S_{\rm max}$ with the topology defined in \cite{BJ:compactifications}, III.11.2. This topology is defined by a convergence class of sequences, from which one obtains a closure operator which in turn induces a topology; see \cite{BJ:compactifications}, p113--114. By \cite{BJ:compactifications}, Proposition III.11.7, the action of $G_\QQ$ on $X$ extends to a continuous action on $_{\QQ}\overline{X}^S_{\rm max}$ and, by \cite{BJ:compactifications}, Theorem III.11.9, the quotient
\begin{align*}
\overline{\Gamma\backslash X}^S_{\rm max}:=\Gamma\backslash_{\QQ}\overline{X}^S_{\rm max},
\end{align*}
endowed with the quotient topology, is a compact Hausdorff space, inside of which $\Gamma\backslash X$ is a dense open subset. We refer to $\overline{\Gamma\backslash X}^S_{\rm max}$ as the {\it maximal Satake compactification} of $\Gamma\backslash X$. 

If $\cE$ is any set of representatives for the proper (rational) parabolic subgroup of $\bf G$ modulo $\Gamma$--conjugation, the maximal Satake compactification $\overline{\Gamma\backslash X}^S_{\rm max}$ is equal to the disjoint union 
\begin{align}\label{satun}
\Gamma\backslash X\coprod \coprod_{{\bf P}\in \cE }\Gamma_{ X_P}\backslash X_P,
\end{align}
where $\Gamma_{X_P}$ is the projection of $\Gamma_P=\Gamma\cap P$ to $M_P$.

\subsection{Homogeneous probability measures on $\overline{\Gamma\backslash X}^S_{\rm max}$}

Consider the situation described in Section \ref{sat} (in particular, $X$ denotes the symmetric space $G/K$, where $K$ is a maximal compact subgroup of $G$). If $\bf H$ is a connected algebraic subgroup of $\bf G$ of type $\mathcal{H}$ and $g\in G$, the homogeneous probability measure on $\Gamma\backslash G$ associated with ${\bf H}$ and $g$ pushes forward to $\overline{\Gamma\backslash X}^S_{\rm max}$ under the natural maps
\begin{align*}
\Gamma\backslash G\rightarrow\Gamma\backslash X\rightarrow\overline{\Gamma\backslash X}^S_{\rm max}.
\end{align*} 
We refer to this probability measure as the {\it homogeneous probability measure on $\overline{\Gamma\backslash X}^S_{\rm max}$ associated with $\bf H$ and $g$}. 

Similarly, if $\bf P$ is a parabolic subgroup of $\bf G$, $\bf H$ is a subgroup of ${\bf P}$ of type $\mathcal{H}$ and $g\in P$, we can define the {\it homogeneous probability measure on $\overline{\Gamma\backslash X}^S_{\rm max}$ associated with $\bf P$, $\bf H$ and $g$} in precisely the same way via the natural maps
\begin{align*}
\Gamma_P\backslash P\rightarrow\Gamma_{X_P}\backslash X_P\rightarrow\overline{\Gamma\backslash X}^S_{\rm max}.
\end{align*}
(Recall that $X_P=M_P/K_P$ and the first map is induced by the projection from $P$ to $M_P$.) We say that a Borel probability measure on $\overline{\Gamma\backslash X}^S_{\rm max}$ is {\it homogeneous} if is a homogeneous probability measure.

\subsection{Properties of Satake compactifications}

Consider the situation described in Section \ref{sat}. We first make two elementary remarks.

\begin{rem}\label{changeK}
Consider another maximal compact subgroup $gKg^{-1}$ of $G$, for some $g\in G$ (recall that all maximal compact subgroups of $G$ are of this form). The maximal Satake compactifications of $\Gamma\backslash X$ corresponding to $K$ and $gKg^{-1}$ are homeomorphic. It follows that Theorem \ref{tI1} is equivalent to the same statement in which $K$ is replaced with $gKg^{-1}$ and the $g_n$ are replaced with $g_ng^{-1}$.
\end{rem}

\begin{rem}\label{changeGamma}
Similarly, for any $c\in G_\QQ$, we obtain a homeomorphism
\begin{align*}
\Gamma\backslash_{\QQ}\overline{X}^S_{\rm max}\rightarrow (c^{-1}\Gamma c)\backslash_{\QQ}\overline{X}^S_{\rm max}
\end{align*}
of compactifications induced by the homeomorphism $x\mapsto cx$ on $_{\QQ}\overline{X}^S_{\rm max}$ (recall that the action is continuous). It follows that Theorem \ref{tI1} is equivalent to the same statement in which we replace $\Gamma$ with $c^{-1}\Gamma c$ and we replace the ${\bf H}_n$ with $c^{-1}{\bf H}_nc$ and the $g_n$ with $c^{-1}g_n$.
\end{rem}

Next we prove a result regarding the structure of the Satake compactifications that is presumably well-known to experts.

\begin{prop}\label{2implies1}
\
\begin{enumerate}
\item[(i)] Let $\Gamma_{X_Q}\backslash X_Q\subset\overline{\Gamma\backslash X}^S_{\rm max}$ be a boundary component as above for some $\gQ\in\cE$. Then the closure of $\Gamma_{X_Q}\backslash X_Q$ in $\overline{\Gamma\backslash X}^S_{\rm max}$ with respect to the above topology is homeomorphic to the maximal Satake compactification of $\Gamma_{X_Q}\backslash X_Q$.
\item[(ii)] Theorem \ref{tI1} (ii) implies Theorem \ref{tI1} (i).
\end{enumerate}
\end{prop}

\begin{proof}
\
\begin{enumerate}
\item[(i)] By Remark \ref{changeK}, we may assume that $M_\gQ$ corresponds to a group $\gM_\gQ$ defined over $\QQ$. Note that $X_Q$ is equal to $M^\der_\gQ/K^\der_Q$, where $\gM^\der_\gQ$ is the derived subgroup of $\gM_\gQ$ and $K^\der_Q$ is $K\cap M^\der_\gQ$. Therefore, to define the maximal Satake compactification, we consider the disjoint union
\begin{align*}
_\QQ\overline{X_Q}^S_{\rm max}=\coprod_{\gP'\subset\gM^\der_\gQ}X_{P'},
\end{align*} 
where $\gP'$ varies over the (rational) parabolic subgroups of $\gM^\der_\gQ$, $X_{P'}=M_{P'}/K_{P'}$, and $K_{P'}=K^\der_Q\cap M_{P'}$. Note that for each parabolic $\gP'$ of $\gM^\der_\gQ$ we have a corresponding parabolic subgroup $\gP$ of $\gG$ contained in $\gQ$ (see \cite{BJ:compactifications}, p276, (III.1.13)) and $X_P=X_{P'}$. In particular, $_\QQ\overline{X_Q}^S_{\rm max}$ is a subset of $_{\QQ}\overline{X}^S_{\rm max}$ and it follows from \cite{BJ:compactifications}, III.11.2 that it is set-theoretically equal to the closure $\overline{X_Q}$ of $X_Q$ in $_{\QQ}\overline{X}^S_{\rm max}$. Furthermore, from the description of convergent sequences given in \cite{BJ:compactifications}, III.11.2, the Satake topology on $_\QQ\overline{X_Q}^S_{\rm max}$ coincides with the induced topology from $_{\QQ}\overline{X}^S_{\rm max}$.

Replace $\Gamma_{X_Q}$ with its intersection with $M^\der_\gQ$. Then we have
\begin{align*}
\Gamma_{X_Q}\backslash X_Q\subset\overline{\Gamma_{X_Q}\backslash X_Q}^S_{\rm max}=\coprod_{\gP'\subset\gM^\der_\gQ} \Gamma_{X_{Q}}\backslash \Gamma_{X_Q}X_{P'}=\coprod_{\gP\subset\gQ} \Gamma\backslash\Gamma X_{P}\subset \overline{\Gamma\backslash X}^S_{\rm max}.
\end{align*}
It follows from the definition of the Satake topology that $\Gamma_{X_Q}\backslash X_Q$ is dense in $\coprod_{\gP\subset\gQ} \Gamma\backslash\Gamma X_{P}$ and that $\coprod_{\gP\subset\gQ} \Gamma\backslash\Gamma X_{P}$ is closed in $\overline{\Gamma\backslash X}^S_{\rm max}$. This proves the result.

\item[(ii)] Let $(\mu_n)_{n\in\NN}$ be a convergent sequence of homogeneous probability measures in $\mathcal{P}(\overline{\Gamma\backslash X}^S_{\rm max})$ supported on a boundary component $\Gamma_{X_Q}\backslash X_Q$. We want to prove that the limit of $(\mu_n)_{n\in\NN}$ is homogeneous. By definition, there exist sequences $(\gH_n)_{n\in\NN}$ and $(g_n)_{n\in\NN}$ with $\gH_n$ a connected algebraic subgroup of $\gQ$ ot type $\mathcal{H}$ and $g_n\in Q$ such that $\mu_n$ is the homogeneous probability measure on $\overline{\Gamma\backslash X}^S_{\rm max}$ associated with $\gQ$, $\gH_n$, and $g_n$. By Remark \ref{changeK}, we may assume that $M_\gQ$ corresponds to a group $\gM_\gQ$ defined over $\QQ$. Since $X_Q$ is equal to $M^\der_\gQ/K^\der_Q$, as in (i), we may replace $\gH_n$ with its image in $\gM^\der_\gQ$ and $g_n$ with its image in $M^\der_\gQ$ and then the result follows from (i).
\end{enumerate}
\end{proof}

\section{Further preliminaries}\label{fprelim}
 
In this section, we collect preliminaries used in the proof of Theorem \ref{tI1}. Many of these also appeared in \cite{DGU}.

\subsection{Rational Langlands decomposition}\label{ratLang}

Let ${\bf G}$ be a connected algebraic group and let $K$ be a maximal compact subgroup of $G$. Let $\bf P$ be a parabolic subgroup of $\bf G$. As in \cite{BJ:compactifications}, (III.1.3), we have the {\it rational Langlands decomposition (with respect to $K$)}
\begin{align*}
P=N_{\bf P}M_{\bf P}A_{\bf P}.
\end{align*}
Since $G=PK$, the rational Langlands decomposition of $P$ yields
\begin{align*}
G=N_{\bf P}M_{\bf P}A_{\bf P}K.
\end{align*}
In particular, if $g\in G$, we can write $g$ as 
\begin{align*}
g=nmak\in N_{\bf P}M_{\bf P}A_{\bf P}K.
\end{align*}
Note that the product $N_{\bf P}M_{\bf P}$ is always associated with a connected algebraic group over $\QQ$, which we denote $\bf H_P$ (see \cite{DGU}, Section 2.6).

\subsection{Standard parabolic subgroups}\label{standard}

Let $\bf G$ be a reductive algebraic group and let $\bf A$ be a maximal split subtorus of $\bf G$. The non-trivial characters of ${\bf A}$ that intervene in the adjoint representation of ${\bf G}$ restricted to ${\bf A}$ are known as the {\it $\QQ$--roots of ${\bf G}$ with respect to ${\bf A}$}.

Let ${\bf P}_0$ be a minimal parabolic subgroup of $\bf G$ containing $\bf A$. We let $\Phi({\bf P}_0,{\bf A})$ denote the set of characters of $\bf A$ occurring in its action on $\mathfrak{n}$, where ${\bf N}={\bf N}_{{\bf P}_0}$. As explained in \cite{BJ:compactifications}, III.1.7, $\Phi({\bf P}_0,{\bf A})$ contains a unique subset $\Delta=\Delta({\bf P}_0,{\bf A})$ such that every element of $\Phi({\bf P}_0,{\bf A})$ is a linear combination, with non-negative integer coefficients, of elements belonging to $\Delta$. On the other hand, ${\bf P}_0$ is determined by $\bf A$ and $\Delta$. We refer to $\Delta$ as a set of {\it simple $\QQ$--roots of $\bf G$ with respect to $\bf A$}.

For a subset $I\subset\Delta$, we define the subtorus
\begin{align*}
{\bf A}_I=(\cap_{\alpha\in I}\ker\alpha)^{\circ}
\end{align*} 
of $\bf A$. Then the subgroup ${\bf P}_I$ of $\bf G$ generated by ${\bf Z}_{\bf G}({\bf A}_I)$ and $\bf N$ is a parabolic subgroup of $\bf G$. We refer to ${\bf P}_I$ as a {\it standard parabolic subgroup} of $\bf G$. Every parabolic subgroup of $\bf G$ containing ${\bf P}_0$ is equal to ${\bf P}_I$ for some uniquely determined subset $I\subset\Delta$.

Let $K$ be a maximal compact subgroup of $G$ such that $A$ is invariant under the Cartan involution of $G$ associated with $K$. Then, as in \cite{DM:uniflows}, Section 1, ${\bf Z}_{\bf G}({\bf A}_I)$ is the Levi subgroup of $\bf P$ appearing in the rational Langlands decomposition of $P$ with respect to $K$. Note that ${\bf A}_I$ is the maximal split subtorus of the centre of ${\bf Z_G(A}_I)$ and we can write ${\bf Z_G(A}_I)$ as an almost direct product ${\bf M}_I{\bf A}_I$, where ${\bf M}_I$ is a reductive group with no rational characters. The rational Langlands decomposition with respect to $K$ is then
\begin{align*}
P_I=N_IM_IA_I,
\end{align*}
where ${\bf N}_I={\bf N}_{{\bf P}_I}$. We will also write ${\bf H}_I={\bf N}_I{\bf M}_I$. For ease of notation, when $I=\Delta\setminus\{\alpha\}$ for some $\alpha\in\Delta$, we will write ${\bf P}_\alpha$, ${\bf A}_\alpha$, ${\bf N}_\alpha$, ${\bf M}_\alpha$, and ${\bf H}_\alpha$ instead of ${\bf P}_I$, ${\bf A}_I$, ${\bf N}_I$, ${\bf M}_I$, and ${\bf H}_I$, respectively.

The set $I$ restricts to a set of simple roots of $\gM_I$ for its maximal split torus $\gA^I=(\gA\cap\gM_I)^\circ$. Therefore, for any subset $J\subset I$, we obtain, as before, a standard parabolic subgroup of $\gM_I$, which we denote $\gP^I_J$. We let $K_I$ denote the maximal compact subgroup $K\cap M_I$ of $M_I$ and we obtain a rational Langlands decomposition with respect to $K_I$:
\begin{align*}
P^I_J=N^I_JM^I_JA^I_J=H^I_JA^I_J,
\end{align*}
where $\gN^I_J=\gN_{\gP^I_J}$, $\gA^I_J=\gA_J\cap \gA^I$ is the maximal split torus in the center of $\gZ_{\gM_I}(\gA^I_J)$, $\gZ_{\gM_I}(\gA^I_J)$ is the almost direct product of $\gA^I_J$ and $\gM^I_J$, and $\gH^I_J=\gN^I_J\gM^I_J$. We will require several elementary lemmas.

\begin{lem}\label{inc}
Let $J\subset I\subset\Delta$. Then $\gA_I\subset\gA_J$, $\gZ_\gG(\gA_J)\subset\gZ_\gG(\gA_J)$, and $\gP_J\subset\gP_I$
\end{lem}

\begin{proof}
The claim $\gA_I\subset\gA_J$ is immediate from the definition. From this we obtain $\gZ_\gG(\gA_J)\subset\gZ_\gG(\gA_I)$. Then $\gP_J\subset\gP_I$ follows immediately.
\end{proof}

\begin{lem}

Let $J\subset I\subset \Delta$. Then $\gN_J=\gN_I \gN_J^{I}$, $\gM_J=\gM_J^{I}$, and $\gA_J=\gA_J^{I} \gA_I$.

\end{lem}

\begin{proof}
We recall that (as in \cite{BJ:compactifications}, (I.1.21)) $\gP_J\subset\gP_I$ is obtained from $\gP_I$ by writing $\gP_I=\gN_I\gM_I\gA_I$ and replacing $\gM_I$ by its parabolic subgroup $\gP^I_J=\gN^I_J\gM^I_J\gA^I_J$. That is, $\gP_J=\gN_I\gN^I_J\gM^I_J\gA^I_J\gA_I$, from which the claims follow.
\end{proof}

The following corollaries are now immediate.

\begin{cor}
Let $J\subset I\subset\Delta$. Then $\gM_J\subset\gM_I$ and $\gN_I\subset\gN_J$.
\end{cor}

\begin{cor}\label{Hdecomp}
Let $J\subset I\subset\Delta$. Then $\gN_I\gH^I_J=\gN_I\gN^I_J\gM^I_J=\gN_J\gM_J=\gH_J\subset\gH_I$.  (Note that the outer equalities are simply the definitions.) 
\end{cor}

\begin{lem}\label{tori}
Let $I_3\subset I_2\subset I_1\subset \Delta$. Then
\begin{align*}
\gA^{I_1}_{I_3}=\gA^{I_2}_{I_3}\gA^{I_1}_{I_2}.
\end{align*}
\end{lem}

\begin{proof}
It is immediate that $\gA^{I_2}$ and $\gA^{I_1}_{I_2}$ are contained in $\gA^{I_1}$. Furthermore, it is easy to show that $\gA^{I_2}\cap\gA^{I_1}_{I_2}$ is finite. Therefore, comparing dimensions, we conclude $\gA^{I_1}=\gA^{I_2}\gA^{I_1}_{I_2}$. Since $\gA^{I_1}_{I_3}=\gA^{I_1}\cap\gA_{I_3}$ and $I_3\subset I_2$, the claim follows.
\end{proof}

\begin{lem}\label{para}
Let $J\subset I\subset\Delta$ and $J'\subset I'\subset\Delta$ such that $I\subset I'$ and $J\subset J'$. Then
\begin{align*}
\gN_I\gP^I_J\subset\gN_{I'}\gP^{I'}_{J'}.
\end{align*}
\end{lem}

\begin{proof}
First observe that, by Lemma \ref{tori}, 
\begin{align*}
\gA^{I'}_J=\gA^{J'}_J\gA^{I'}_{J'}=\gA^I_J\gA^{I'}_I.
\end{align*}
Now, we have
\begin{align*}
\gN_{I'}\gP^{I'}_{J'}=\gN_{I'}\gN^{I'}_{J'}\gM^{I'}_{J'}\gA^{I'}_{J'}=\gH_{J'}\gA^{I'}_{J'}=\gH_{J'}\gA^{J'}_J\gA^{I'}_{J'}=\gH_{J'}\gA^{I'}_J,
\end{align*}
where the second equality is Corollary \ref{Hdecomp} and the third equality follows from the fact that $\gA^{J'}\subset\gH_{J'}$. Therefore, since $\gH_J\subset\gH_{J'}$ (by Corollary \ref{Hdecomp}) and $\gA^I_J\subset\gA^{I'}_J$, we obtain 
\begin{align*}
\gN_I\gP^I_J=\gN_I\gN^I_J\gM^I_J\gA^I_J=\gH_J\gA^I_J\subset\gN_{I'}\gP^{I'}_{J'}
\end{align*}
as claimed (using Corollary \ref{Hdecomp} for the second equality).
\end{proof}

\subsection{The $d_{\gP,K}$ functions}\label{dfunctions} 
Let ${\bf G}$ be a reductive algebraic group and let $K$ be a maximal compact subgroup of $G$. Let $\bf P$ be a proper parabolic subgroup of $\bf G$ and let $n_{\bf P}$ denote the dimension of $\mathfrak{n}_{\bf P}$. Consider the $n_{\bf P}^{\rm th}$ exterior product $V_{\bf P}=\wedge^{n_{\bf P}}\mathfrak{g}$ of $\mathfrak{g}$ and let $L_{\bf P}$ denote the one-dimensional subspace given by $\wedge^{n_{\bf P}}\mathfrak{n}_{\bf P}$. Then the adjoint representation induces a linear representation of $\bf G$ on $V_{\bf P}$ and, since $\bf P$ normalizes $\bf N_P$, we obtain a linear representation of $\bf P$ on $L_{\bf P}$. That is, $\bf P$ acts on $L_{\bf P}$ via a character $\chi_{\bf P}$.

Fix a $K$--invariant norm $\|\cdot\|_{\bf P}$ on $V_{\bf P}\otimes_{\QQ}\RR$ and let $v_{\bf P}\in L_{\bf P}\otimes_{\QQ}\RR$ be such that $\|v_{\bf P}\|_{\bf P}=1$. We obtain a function $d_{{\bf P},K}$ on $G$ defined by
\begin{align*}
d_{{\bf P},K}(g)=\|g\cdot v_{\bf p}\|_{\bf P}.
\end{align*}
Note that, for any $g\in G$, we can write $g=kp$, where $k\in K$ and $p\in P$. Therefore,
\begin{align*}
d_{{\bf P},K}(g)=\|g\cdot v_{\bf p}\|_{\bf P}=\|p\cdot v_{\bf p}\|_{\bf P}=\chi_{\bf P}(p)\cdot\|v_{\bf p}\|_{\bf P}=\chi_{\bf P}(p)
\end{align*}
(note that $\chi_{\bf P}$ is necessarily positive on the connected component $P$).
In particular, $d_{{\bf P},K}$ is a {\it function on $G$ of type $(P,\chi_{\bf P})$}, as defined in \cite{borel:arithmetic-groups}, Section 14.1. Furthermore, it does not depend on the choices of $\|\cdot\|_{\bf P}$ and $v_{\bf P}$.

\begin{rem}
Suppose that $\Gamma\subset\gG(\QQ)$ is an arithmetic subgroup. Let $\Gamma_P$ denote $\Gamma\cap P$ and let $\Gamma_{H_\gP}$ denote $\Gamma\cap H_\gP$. We claim that $\Gamma_P=\Gamma_{H_\gP}$. To see this, write $\gP$ as the product of $\gH_\gP$ with a maximal split subtorus $\gS_\gP$ of the centre of a Levi subgroup of $\gP$ (defined over $\QQ$ but not necessarily arising from the rational Langlands decomposition associated with $K$). Let $\Delta(\gS_\gP)$ denote a basis for the character group of $\gS_\gP$ and, for each $\chi\in\Delta(\gS_\gP)$, let $\tilde{\chi}$ denote the morphism $\gP\to\GG_m$ defined by $(h,s)\mapsto\chi(s)^n$ for some sufficiently large $n\in\NN$ such that the morphism is well-defined. By \cite{borel:arithmetic-groups}, Section 8.11, $\tilde{\chi}(\Gamma_P)$ is an arithmetic subgroup of $\QQ^\times$ for any $\chi\in\Delta(\gS_\gP)$. On the other hand, it is contained in $\RR_{>0}$ and, therefore, it is trivial. Hence, $\Gamma_P$ is contained in 
\begin{align*}
\gG_{\Delta(\gS_\gP)}=\cap_{\chi\in\Delta(\gS_\gP)}\ker\tilde{\chi}
\end{align*}
and, since $G_{\Delta(\gS_\gP)}=H_\gP$, the claim follows.
\end{rem}

\subsection{The $\delta$--functions}\label{deltafunctions}

Let $\bf G$ be a connected algebraic group with no rational characters and let $\bf L$ be a Levi subgroup of $\bf G$. Then $\bf G$ is the semidirect product of $\bf L$ and ${\bf N}={\bf N}_{\bf G}$. We denote by $\pi$ the natural (surjective) morphism from ${\bf G}$ to ${\bf L}$.

Let ${\bf P}_0$ be a minimal parabolic subgroup of $\bf L$ and let $\bf A$ be a maximal split subtorus of $\bf L$ contained in ${\bf P}_0$. Let $K$ denote a maximal compact subgroup of $L$ such that $A$ is invariant under the Cartan involution of $G$ associated with $K$. For any proper parabolic subgroup $\bf P$ of $\bf L$, we obtain a function $d_{{\bf P},K}$ on $L$, as defined in Section \ref{dfunctions}, and, for each $\alpha\in\Delta=\Delta({\bf P}_0,{\bf A})$, we write $d_\alpha=d_{{\bf P}_\alpha,K}$.

Let $\Gamma$ be an arithmetic subgroup of $G_\QQ$ and let $\Gamma_L=\pi(\Gamma)$, which is arithmetic subgroup of $L_\QQ$. By \cite{borel:arithmetic-groups}, Th\'eor\`eme 13.1, there exists a finite subset $F$ of $L_\QQ$ and a $t>0$ such that $L=KA_t\omega F^{-1}\Gamma_L$,
where $\omega$ is a compact subset of $H_{{\bf P}_0}$ and 
\begin{align*}
A_t=\{a\in A:\alpha(a)\leq t\text{ for all }\alpha\in\Delta\}.
\end{align*}
As in \cite{DGU}, we refer to a set $F$ as above as a $\Gamma_L$--set for $\bf L$. 

For any connected algebraic subgroup $\bf H$ of $\bf G$ and any $g\in G$, we define
\begin{align*}
\delta(\gG,K,\Delta,F)({\bf H},g)={\rm inf}\{d_\alpha(\pi(g)^{-1}\lambda):\lambda\in\Gamma_LF,\ \alpha\in\Delta,\ {\bf H}\subset{\bf N}\lambda{\bf P}_{\alpha}\lambda^{-1}\}
\end{align*}
(where we take the value to be $\infty$ if the infimum is varying over the empty set).
By \cite{DGU}, Lemma 4.3, we have $\delta(\gG,K,\Delta,F)({\bf H},g)>0$. 

\section{Main tools}\label{tools}

In this section, we describe the three main tools used in the proof of Theorem \ref{tI1}.

\subsection{The criterion for convergence in $\Gamma\backslash G$}

Consider the situation described in Section \ref{deltafunctions}. Let $\delta=\delta(\gG,K,\Delta,F)$. The following result is the combination of \cite{DGU}, Theorem 4.6 and \cite{DGU}, Theorem 2.9. However, it should be emphasized that it is a very modest generalization of a result of Eskin--Mozes--Shah \cite{EMS:flows}, making similar use of the same tools, namely, those of Dani--Margulis, Eskin--Mozes--Shah, Mozes--Shah, and Ratner (see \cite{DM:uniflows}, \cite{EMS:flows}, \cite{EMS:nondivergence}, \cite{MS:ergodic}, and \cite{ratner:measure}).

\begin{teo}\label{criterion1}
For each $n\in\NN$, let ${\bf H}_n$ be a connected algebraic subgroup of $\bf G$ of type $\mathcal{H}$, let $g_n\in G$ and let $\mu_n$ be the homogeneous probability measure on $\Gamma\backslash G$ associated with ${\bf H}_n$ and $g_n$. Assume that
\begin{align*}
\liminf_{n\rightarrow\infty}\ \delta({\bf H}_n,g_n)>0.
\end{align*}
Then the set $\{\mu_n\}_{n\in\NN}$ is sequentially compact in $\mathcal{P}(\Gamma\backslash G)$. 

Furthermore, if $\mu$ is a limit point in $\{\mu_n\}_{n\in\NN}$, then $\mu$ is the homogeneous probability measure on $\Gamma\backslash G$ associated with a connected algebraic subgroup ${\bf H}$ of $\bf G$ of type $\mathcal{H}$ and an element $g\in G$, and ${\bf H}_n$ is contained in $\bf H$ for all $n$ large enough. 
\end{teo}

\subsection{The criterion for convergence in $\overline{\Gamma\backslash X}^S_{\rm max}$}

Consider the situation described in Section \ref{sat}. For a subgroup $H$ of $G$, we set $\Gamma_H=\Gamma\cap H$. The following result is \cite{DGU}, Theorem 5.1.

\begin{teo}\label{criterion2}
For each $n\in\NN$, let ${\bf H}_n$ denote a connected algebraic subgroup of $\bf G$ of type $\mathcal{H}$, let $g_n$ denote an element of $G$ and let $\mu_n$ denote the homogeneous probability measure on $\overline{\Gamma\backslash X}^S_{\rm max}$ associated with ${\bf H}_n$ and $g_n$.

Suppose that there exists a parabolic subgroup ${\bf P}$ of $\bf G$ such that,
\begin{enumerate}
\item[(i)] for all $n\in\NN$, ${\bf H}_n$ is contained in $\bf H_P$,
\item[(ii)] we can write
\begin{align*}
g_n=h_na_nk_n\in H_{\bf P}A_{\bf P}K,
\end{align*}	
such that
\begin{align*}
\alpha(a_n)\rightarrow\infty,\text{ as }n\rightarrow\infty,\text{ for all }\alpha\in\Phi(P,A_{\bf P}),
\end{align*}
and, 
\item[(iii)] if we denote by $\nu_n$ the homogeneous probability measure on $\Gamma_{H_{\bf P}}\backslash H_{\bf P}$ associated with ${\bf H}_n$ and $h_n$, then $(\nu_n)_{n\in\NN}$ converges to $\nu\in\mathcal{P}(\Gamma_{H_{\bf P}}\backslash H_{\bf P})$.
\end{enumerate}

Then there exists a connected algebraic subgroup $\bf H$ of ${\bf P}$ of type $\mathcal{H}$ and an element $g\in P$ such that $(\mu_n)_{n\in\NN}$ converges to the homogeneous probability measure on $\overline{\Gamma\backslash X}^S_{\rm max}$ associated with $\bf P$, $\bf H$ and $g$, and, furthermore, ${\bf H}_n$ is contained in ${\bf H}$ for $n$ large enough.
\end{teo}

Note that, by the results of Mozes--Shah \cite{MS:ergodic} and Ratner \cite{ratner:measure} condition (iii) is equivalent to the stronger statement that $\nu$ not only exists but is also homogeneous (see \cite{DGU}, Theorem 2.9). This is used crucially in the proof of Theorem \ref{criterion2} (see \cite{DGU}).

\subsection{A key result on root systems}\label{keyresult}

Let $\gG$ be reductive algebraic group, let $\gP_0$ be a minimal parabolic subgroup of $\gG$, and let $\gA$ be a maximal split subtorus of $\gG$ contained in $\gP_0$. Let $X^*({\bf A})$ denote the character module of $\bf A$ and let $X^*({\bf A})_{\QQ}$ denote the $\QQ$--vector space $X^*({\bf A})\otimes_{\ZZ}\QQ$. Fix a non-degenerate scalar product $(\cdot,\cdot)$ on $X^*({\bf A})_{\QQ}$ that is invariant under the action of ${\bf N}_{\bf G}({\bf A})(\QQ)$. Then the $\QQ$--roots of ${\bf G}$ with respect to ${\bf A}$  equipped with the inner product $(\cdot,\cdot)$ constitute a {\it root system} in $X^*({\bf A})_{\QQ}$. We refer the reader to \cite{springer:redgroups}, Section 3.5 for further details.

Let $\Delta=\Delta(\gP_0,\gA)$ and, for each $\alpha\in\Delta$, let $w_\alpha\in X^*(\gA)_\QQ$ be the unique element such that $(w_\alpha,\beta)=\delta_{\alpha\beta}$ for any $\beta\in\Delta$. These elements are usually called {\it dual weights}. They are a particular choice of quasi-fundamental weights, as defined in \cite{DGU}, Section 2.9. It follows from \cite{DGU}, Lemma 4.1 that $w_\alpha$ is a positive rational multiple of the character $\chi_{\alpha}$ defined therein (namely, the restriction of $\chi_{{\bf P}_\alpha}$ to $\bf A$). We therefore deduce the following fact.

\begin{lem}\label{trivial}
Let $\alpha\in\Delta$ and $I=\Delta\setminus\{\alpha\}$. Then $w_\alpha$ is trivial on $\gA^I$.
\end{lem}

\begin{proof}
Since $\gH_I$ has no rational characters and $\gA^I$ is contained in $\gH_I$, the result follows from the fact that $\chi_\alpha$ is the restriction of a character on $\gP_I$.
\end{proof}

Notice that $\Delta$ and $\{w_\alpha\}_{\alpha\in\Delta}$ constitute two bases of $X^*({\bf A})_{\QQ}$. By construction, we have 
\begin{align*}
w_\alpha=\sum_{\beta\in\Delta}(w_\alpha,w_\beta)\beta\hspace{1cm}\text{and}\hspace{1cm}
\alpha=\sum_{\beta\in\Delta}(\alpha,\beta)w_\beta,
\end{align*}
and it is a simple calculation to see that the matrices with coefficients $(w_\alpha,w_\beta)$ and $(\alpha,\beta)$, respectively, are inverse to one another. It follows from \cite{LGLA}, Chapter 3, Proposition 1.16 that $(w_\alpha,w_\beta)\geq 0$ for all $\alpha,\beta\in\Delta$ (see also the paragraph following Lemma 6.5 in the appendix).
We define $d_\alpha=\sum_{\beta\in\Delta}(w_\alpha,w_\beta)$. In particular, $d_\alpha>0$. 

\begin{defini}\label{wfw}
For each $\alpha\in\Delta$, we write $\bar{w}_\alpha=\frac{1}{d_\alpha}w_\alpha$. We refer to the set $\{\bar{w}_\alpha\}_{\alpha\in\Delta}$ as a set of {\it weighted} dual weights of $\gG$ with respect to $\gA$. From the above and \cite{DGU}, Section 4.1 (3), we see that $\bar{w}_\alpha$ is a positive rational combination of positive roots.
\end{defini}

The key result we need is the following. A proof (due to Jialun Li) is provided in the appendix (see Corollary \ref{cor:roots}).

\begin{teo}[Li]\label{rootconj}
Let $\alpha\in\Delta$ and let $I\subset\Delta\setminus\{\alpha\}$ be such that $\alpha$ is connected to $\Delta\setminus(I\cup\{\alpha\})$ in the Dynkin diagram. For each $n\in\NN$, let $a_n\in A_I$ and suppose that
\begin{align*}
\beta(a_n)\rightarrow\infty,\text{ as }n\rightarrow\infty
\end{align*}
for all $\beta\in\Delta\setminus(I\cup\{\alpha\})$. Furthermore, suppose that 
\begin{align*}
\bar{w}_{\alpha}(a_n)\geq\bar{w}_\beta(a_n)\text{ for all }\beta\in\Delta\setminus I\hspace{1cm}\text{and}\hspace{1cm}\bar{w}_{\alpha}(a_n)\geq 1.
\end{align*}
Then
\begin{align*}
\alpha(a_n)\rightarrow\infty,\text{ as }n\rightarrow\infty.
\end{align*}
\end{teo}

Note that, in the appendix, the condition $\bar{w}_{\alpha}(a_n)\geq 0$ appears. The calculations there take place in the Lie algebra $\mathfrak{a}_I$ of $A_I$. Passing to $A_I$ is via the exponential map, hence the condition  $\bar{w}_{\alpha}(a_n)\geq 1$ above.

The condition ``$\alpha$ is connected to $\Delta\setminus(I\cup\{\alpha\})$ in the Dynkin diagram'' is to say that if we write $\Delta$ as the disjoint union $\Delta_1\cup\cdots\cup\Delta_m$ according to the decomposition of the root system into irreducible root systems, and $\alpha\in\Delta_k$, then, among the roots in $\Delta\setminus(I\cup\{\alpha\})$, there exists at least one root that also belongs to $\Delta_k$.  In other words, $I\cup\{\alpha\}$ does not contain $\Delta_k$.

\begin{lem}\label{discon}
Let $\alpha\in\Delta$ and let $I\subset\Delta\setminus\{\alpha\}$ be such that $\alpha$ is not connected to $\Delta\setminus(I\cup\{\alpha\})$ in the Dynkin diagram. Let $J\subset \Delta$. Then 
\begin{align*}
\gA^J_{(I\cup\{\alpha\})\cap J}\subset\ker\alpha.
\end{align*}.
\end{lem}

\begin{proof}
Let $\gG^\ad$ denote the quotient of $\gG$ by its center, and let $\ad:\gG\rightarrow\gG^\ad$ denote the natural morphism. Recall that the adjoint representation factors through $\ad$. Furthermore, $\gG^\ad$ is equal to a product $\gG_1\times\cdots\times\gG_m$ of $\QQ$--simple groups, $\gP^\ad_0=\ad(\gP_0)$ is a product $\gP_1\times\cdots\times\gP_m$ of minimal parabolic subgroups, and $\gA^\ad=\ad(\gA)$ is a product $\gA_1\times\cdots\times\gA_m$ of maximal split tori. The decomposition $\Delta=\Delta_1\cup\cdots\cup\Delta_m$ above is
\begin{align*}
\Delta=\Delta(\gP^\ad_0,\gA^\ad)=\Delta(\gP_1,\gA_1)\cup\cdots\cup\Delta(\gP_m,\gA_m).
\end{align*}
The image of $\gA^J$ in $\gA^\ad$ is therefore equal to the product $\gA^{J_1}_1\times\cdots\times\gA^{J_m}_m$, where $J_i=J\cap\Delta_i$. The assumption that $\alpha$ is not connected to $\Delta\setminus(I\cup\{\alpha\})$ in the Dynkin diagram is equivalent to the statement that $I\cup\{\alpha\}$ contains $\Delta_k$. Therefore, the image of $\gA^J_{(I\cup\{\alpha\})\cap J}$ in $\gA^\ad$ is also a product and its $k^{\rm th}$ component is $\gA^{J_k}_{k,J_k}=\{1\}$. Since $\alpha\in\Delta_k$ is trivial outside of $\gA_k$, we conclude that it is trivial on $\gA^J_{(I\cup\{\alpha\})\cap J}$.
\end{proof}

\section{The proof}\label{proof}

In this section, we will give the proof of Theorem \ref{tI1}. First we will make a useful definition.

\subsection{Maximal couples}

Consider the situation described in Section \ref{deltafunctions}. For each $n\in\NN$, let $\gH_n$ be a connected algebraic subgroup of $\gG$ of type $\mathcal{H}$ and let $g_n\in G$.

\begin{defini}
A maximal couple for $(\gH_n)_{n\in\mathbb{N}}$ and $(g_n)_{n\in\mathbb{N}}$ with respect to $\gG$, $K$, $\Delta$, and $F$ is a couple $(\alpha,(\lambda_n)_{n\in\mathbb{N}})$ for some $\alpha\in\Delta$ and $\lambda_n\in\Gamma_LF$ such that
\begin{align*}
\gH_n\subset\gN\lambda_n\gP_\alpha(\lambda_n)^{-1}
\end{align*}
for every $n\in\mathbb{N}$ and, whenever $\gH_n\subset\gN\lambda'_n\gP_\beta(\lambda'_n)^{-1}$ for some $\beta\in\Delta$ and $\lambda'_n\in\Gamma_LF$ and we write
\begin{align*}
\lambda^{-1}_ng_n=h_na_nk_n\in NH_\alpha A_\alpha K\hspace{.5cm}\text{and}\hspace{.5cm}(\lambda'_n)^{-1}g_n=h'_na'_nk'_n\in NH_\beta A_\beta K
\end{align*}
according to the rational Langlands decompositions, then
\begin{align*}
\bar{w}_\alpha(a_n)\geq\bar{w}_\beta(a'_n),
\end{align*}
where $\bar{w}_\alpha$ and $\bar{w}_\beta$ are weighted dual weights for $\bf L$, as in Definition \ref{wfw}.
\end{defini}

We have the following important properties of maximal couples.

\begin{prop}\label{mc}
\
\begin{enumerate}
\item[(i)] After possibly extracting subsequences, a maximal couple for $(\gH_n)_{n\in\mathbb{N}}$ and $(g_n)_{n\in\mathbb{N}}$ with respect to $\gG$, $K$, $\Delta$, and $F$ exists.
\item[(ii)] Suppose that $\liminf_{n\rightarrow\infty}\ \delta({\bf H}_n,g_n)=0$. If $(\alpha,(\lambda_n)_{n\in\mathbb{N}})$ is a maximal couple for $(\gH_n)_{n\in\mathbb{N}}$ and $(g_n)_{n\in\mathbb{N}}$ with respect to $\gG$, $K$, $\Delta$, and $F$ and we write
\begin{align*}
\lambda^{-1}_ng_n=h_na_nk_n\in NH_\alpha A_\alpha K
\end{align*}
according to the rational Langlands decomposition, then
\begin{align*}
\alpha(a_n)\rightarrow\infty,\text{ as }n\rightarrow\infty.
\end{align*}
\end{enumerate}
\end{prop}

\begin{proof}
\
\begin{enumerate}
\item[(i)] As explained in Section \ref{keyresult}, each $\bar{w}_\beta$ is a positive rational multiple of the corresponding $d_\beta$. Since there are only finitely many of them, the claim follows from \cite{DGU}, Lemma 4.3.
\item[(ii)] Since $\liminf_{n\rightarrow\infty}\ \delta({\bf H}_n,g_n)=0$, after possibly extracting a subsequence, there exists $\beta\in\Delta$ and $(\lambda'_n)_{n\in\mathbb{N}}$ with $\lambda'_n\in\Gamma_LF$ such that $\gH_n\subset\gN\lambda'_n\gP_\beta(\lambda'_n)^{-1}$ and
\begin{align*}
d_\beta(g^{-1}_n\lambda'_n)\rightarrow 0,\text{ as }n\rightarrow\infty.
\end{align*}
Therefore, by \cite{DGU}, Lemma 4.2, if we write
\begin{align*}
(\lambda'_n)^{-1}g_n=h'_na'_nk'_n\in NH_\beta A_\beta K,
\end{align*}
according to the rational Langlands decomposition, we have
\begin{align*}
\beta(a'_n)\rightarrow\infty,\text{ as }n\rightarrow\infty. 
\end{align*}
Now, since $(\alpha,(\lambda_n)_{n\in\mathbb{N}})$ is a maximal couple for $(\gH_n)_{n\in\mathbb{N}}$ and $(g_n)_{n\in\mathbb{N}}$ with respect to $\gG$, $K$, $\Delta$, and $F$, we have
\begin{align*}
\alpha(a_n)=\bar{w}_\alpha(a_n)^{\frac{1}{n_\alpha}}\geq\bar{w}_\beta(a'_n)^{\frac{1}{n_\alpha}}=\beta(a'_n)^{\frac{n_\beta}{n_\alpha}},
\end{align*}
for some $n_\alpha,n_\beta>0$ (see the description of weighted dual weights in Section \ref{keyresult}). Therefore,
\begin{align*}
\alpha(a_n)\rightarrow\infty,\text{ as }n\rightarrow\infty.
\end{align*}
\end{enumerate}
\end{proof}

\subsection{Proof of Theorem \ref{tI1}}
Let ${\bf P}_0$ be a minimal parabolic subgroup of $\gG$ and let $\bf A$ be a maximal split subtorus of $\gG$ contained in ${\bf P}_0$. By Remark \ref{changeK}, we may assume that $A$ is invariant under the Cartan involution of $G$ associated with $K$. Let $I_0=\Delta=\Delta(\gP_0,\gA)$. 

By Proposition \ref{2implies1} (ii), it suffices to prove Theorem \ref{tI1} (ii). By Remark \ref{changeGamma} and Theorem \ref{criterion2}, it suffices to prove the following theorem.

\begin{teo}\label{DGUtype}
Consider the situation described in Theorem \ref{tI1} (ii). After possibly extracting a subsequence, there exists a parabolic subgroup ${\bf P}$ of $\bf G$, an element $c\in G_\QQ$, and, for each $n\in\NN$, an element $\gamma_n\in\Gamma$ such that,
\begin{enumerate}
\item[(i)] for all $n\in\NN$, $c^{-1}\gamma^{-1}_n{\bf H}_n\gamma_nc$ is contained in $\bf H_P$,
\item[(ii)] we can write
\begin{align*}
c^{-1}\gamma^{-1}_ng_n=h_na_nk_n\in H_{\bf P}A_{\bf P}K,
\end{align*}	
such that
\begin{align*}
\alpha(a_n)\rightarrow\infty,\text{ as }n\rightarrow\infty,\text{ for all }\alpha\in\Phi(P,A_{\bf P}),
\end{align*}
and, 
\item[(iii)] if we denote by $\nu_n$ the homogeneous probability measure on $(c^{-1}\Gamma c\cap {H_{\bf P}})\backslash H_{\bf P}$ associated with $c^{-1}\gamma^{-1}_n{\bf H}_n\gamma_nc$ and $h_n$, then $(\nu_n)_{n\in\NN}$ converges to $\nu\in\mathcal{P}((c^{-1}\Gamma c\cap {H_{\bf P}})\backslash H_{\bf P})$.
\end{enumerate}
\end{teo}

\subsection{Outline of the proof of Theorem \ref{DGUtype}}

The essence of the proof is to select successive maximal couples. At each stage, up to conjugation and taking subsequences, we have $\gH_n\subset\gH_I$ and $h_{I,n}\in H_I$ for some fixed $I$. We apply the criterion given by Theorem \ref{criterion1} and, if it fails, we obtain, by Proposition \ref{mc}, a maximal couple $(\alpha,(\lambda_n)_{n\in\NN})$. After conjugation and taking a subsequence, we find $\gH_n\subset\gH_{I\cup\alpha}$ and $h_{I\cup\alpha,n}\in H_{I\cup\alpha}$ and we repeat the procedure. After a finite number of steps, the criterion given by Theorem \ref{criterion1} holds, which guarantees (iii), namely, convergence of the associated homogeneous measures in a space of the form $\Gamma_{\gH_J}\backslash H_J$ for some $J$. The fact that we have chosen maximal couples allows us to apply Theorem \ref{rootconj}, which in turn yields (ii). 

\subsection{Proof of Theorem \ref{DGUtype}}

The proof is via an iterative procedure.

\subsection*{Step $1$} Let $F_0$ denote a $\Gamma$--set for $\gG$ and let $\delta=\delta(\gG,K,\Delta,F_0)$, as defined in Section \ref{deltafunctions}. Note that, if 
\begin{align*}
\liminf_{n\rightarrow\infty}\ \delta({\bf H}_n,g_n)>0,
\end{align*}
then Theorem \ref{DGUtype} follows (with $\gP=\gG$ and $c=\gamma_n=1$) from Theorem \ref{criterion1}. Therefore, assume otherwise. By Proposition \ref{mc} (i), after possibly extracting a subsequence, there exists a maximal couple $(\alpha_{i_1},(\lambda^{(1)}_n)_{n\in\NN})$ for $(\gH_n)_{n\in\mathbb{N}}$ and $(g_n)_{n\in\mathbb{N}}$ with respect to $\gG$, $K$, $\Delta$, and $F_0$. Therefore, if we write $I_1=\Delta\setminus\{\alpha_{i_1}\}$, then
\begin{align*}
\gH_n\subset\lambda^{(1)}_n\gP_{I_1}(\lambda^{(1)}_n)^{-1}
\end{align*}
and, by Proposition \ref{mc} (ii), if we write
\begin{align*}
(\lambda^{(1)}_n)^{-1}g_n=h^{(1)}_na^{(1)}_nk^{(1)}_n\in H_{I_1}A_{I_1}K,
\end{align*}
according to the rational Langlands decomposition, we have
\begin{align*}
\alpha_{i_1}(a^{(1)}_n)\rightarrow\infty,\text{ as }n\rightarrow\infty. 
\end{align*}

\subsection*{Step $2$}
We define 
\begin{align*}
\gH^{(1)}_n=(\lambda^{(1)}_n)^{-1}\gH_n\lambda^{(1)}_n
\end{align*}
and we write $\lambda^{(1)}_n=\gamma^{(1)}_nc_0\in\Gamma F_0$. Then $\gH^{(1)}_n\subset\gH_{I_1}$ because $\gH^{(1)}_n$ has no rational characters. 
We let 
\begin{align*}
\tilde{\Gamma}_1=(\lambda^{(1)}_n)^{-1}\Gamma \lambda^{(1)}_n\cap H_{I_1}=c^{-1}_0\Gamma c_0\cap H_{I_1}
\end{align*}
and we write $\Gamma_1=\pi_1(\tilde{\Gamma}_1)$, where $\pi_1:\gH_{I_1}\rightarrow\gM_{I_1}$ denotes the natural projection. We let $F_1$ denote a $\Gamma_1$--set for $\gM_{I_1}$ and we let $K_{I_1}=K\cap M_{I_1}$. Recall that $I_1$ restricts to a set of simple $\QQ$--roots for $\gM_{I_1}$ with respect to $\gA^{I_1}$. We let $\delta_1=\delta(\gH_{I_1},K_{I_1},I_1,F_1)$ and we apply Theorem \ref{criterion1} to the $\gH^{(1)}_n\subset\gH_{I_1}$ and $h^{(1)}_n\in H_{I_1}$. If
\begin{align*}
\liminf_{n\rightarrow\infty}\ \delta_1(\gH^{(1)}_n,h^{(1)}_n)>0
\end{align*}
then Theorem \ref{DGUtype} follows with $\gP=\gP_{I_1}$ and $c=c_0$ and $\gamma_n=\gamma^{(1)}_n$. Indeed, condition (iii) follows from Theorem \ref{criterion1} and condition (ii) was established in Step $1$.

Therefore, we assume otherwise. By Proposition \ref{mc} (i), after possibly extracting a subsequence, there exists a maximal couple $(\alpha_{i_2},(\lambda^{(2)}_n)_{n\in\NN})$ for $(\gH^{(1)}_n)_{n\in\mathbb{N}}$ and $(h^{(1)}_n)_{n\in\mathbb{N}}$ with respect to $\gH_{I_1}$, $K_{I_1}$, $I_1$, and $F_1$. Therefore,  if we write $I_2=I_1\setminus\{\alpha_{i_2}\}=\Delta\setminus\{\alpha_{i_1},\alpha_{i_2}\}$, then
\begin{align*}
\gH^{(1)}_n\subset\gN_{I_1}\lambda^{(2)}_n\gP^{I_1}_{I_2}(\lambda^{(2)}_n)^{-1}
\end{align*}
and, by Proposition \ref{mc} (ii), if we write
\begin{align*}
(\lambda^{(2)}_n)^{-1}h^{(1)}_n=h^{(2)}_na^{(2)}_nk^{(2)}_n\in N_{I_1}H^{I_1}_{I_2}A^{I_1}_{I_2}K_{I_1}=H_{I_2}A^{I_1}_{I_2}K_{I_1}
\end{align*}
(the equality justified by Corollary \ref{Hdecomp}), then 
\begin{align*}
\alpha_{i_2}(a^{(2)}_n)\rightarrow\infty,\text{ as }n\rightarrow\infty.
\end{align*}

We define
\begin{align*}
\gH^{(2)}_n=(\lambda^{(1)}_n\lambda^{(2)}_n)^{-1}\gH_n\lambda^{(1)}_n\lambda^{(2)}_n
\end{align*}
and we write $\lambda^{(2)}_n=\gamma^{(2)}_nc_1\in\Gamma_1F_1$. After possibly enlarging $F_1$ and extracting a subsequence, we may assume that $\gamma^{(2)}_n\in c_0^{-1}\Gamma c_0\cap M_{I_1}$ (indeed, the latter is an arithmetic subgroup of $M_{I_1}$ contained, and hence finite index, in $\Gamma_1$). 

We have $\gH^{(2)}_n\subset\gN_{I_1}\gH^{I_1}_{I_2}=\gH_{I_2}$ as before (the equality justified by Corollary \ref{Hdecomp}). We let 
\begin{align*}
\tilde{\Gamma}_2=c^{-1}_1\tilde{\Gamma}_1c_1\cap H_{I_2}=c^{-1}_1c^{-1}_0\Gamma c_0c_1\cap H_{I_2}
\end{align*}
and we write $\Gamma_2=\pi_2(\tilde{\Gamma_2})$, where $\pi_2:\gH_{I_2}\rightarrow\gM_{I_2}$ denotes the natural projection. We let $F_2$ denote a $\Gamma_2$--set for $\gM_{I_2}$.

Now we iterate the following step for $r\geq 2$.

\subsection*{Step $r+1$}

We start with groups
\begin{align*}
\gH^{(r)}_n=(\lambda^{(1)}_n\lambda^{(2)}_n\cdots\lambda^{(r)}_n)^{-1}\gH_n\lambda^{(1)}_n\lambda^{(2)}_n\cdots\lambda^{(r)}_n\subset\gH_{I_{r}}
\end{align*}
and elements
\begin{align*}
(\lambda^{(r)}_n)^{-1}\cdots(\lambda^{(2)}_n)^{-1}(\lambda^{(1)}_n)^{-1}g_n&=(\lambda^{(r)}_n)^{-1}\cdots(\lambda^{(3)}_n)^{-1}(\lambda^{(2)}_n)^{-1}h^{(1)}_na^{(1)}_nk^{(1)}_n\\
&=(\lambda^{(r)}_n)^{-1}\cdots(\lambda^{(3)}_n)^{-1}h^{(2)}_n a^{(2)}_n a^{(1)}_n k^{(2)}_n k^{(1)}_n\\
&=h^{(r)}_na^{(r)}_n\cdots a^{(1)}_n k^{(r)}_n\cdots k^{(1)}_n,
\end{align*}
where $h^{(l)}_n\in H_{I_l}$, $a^{(l)}_n\in A^{I_{l-1}}_{I_l}$, $k^{(l)}_n\in K_{I_{l-1}}=K\cap M_{I_{l-1}}$, and $\lambda^{(l)}_n=\gamma^{(l)}_nc_{l-1}\in\Gamma_{l-1}F_{l-1}$, where $I_{l-1}=\Delta\setminus\{\alpha_{i_1},\ldots,\alpha_{i_{l-1}}\}$,
\begin{align*}
\tilde{\Gamma}_{l-1}=c^{-1}_{l-2}\cdots c^{-1}_0\Gamma c_0\cdots c_{l-2}\cap H_{I_{l-1}},
\end{align*} 
$\Gamma_{l-1}=\pi_{l-1}(\tilde{\Gamma}_{l-1})$, where $\pi_{l-1}:\gH_{I_{l-1}}\rightarrow\gM_{I_{l-1}}$ is the natural projection, and $F_{l-1}$ is a $\Gamma_{l-2}$--set for $\gM_{I_{l-1}}$, and $c_l\in F_l$. As before, we may assume that 
\begin{align*}
\gamma^{(l)}_n\in c^{-1}_{l-2}\cdots c^{-1}_0\Gamma c_0\cdots c_{l-2}\cap M_{I_{l-1}}.
\end{align*}

Above, we have made repeated use of the fact that $k^{(l)}_n\in M_{I_{l-1}}$ and so, for $j<l$, it commutes with 
\begin{align*}
a^{(l-j)}_n\in A^{I_{l-j-1}}_{I_{l-j}}\subset A_{I_{l-j}}\subset A_{I_{l-1}}
\end{align*}
(the latter inclusion justified by Lemma \ref{inc}). 

We let 
\begin{align*}
\tilde{\Gamma}_r=c^{-1}_{r-1}\cdots c^{-1}_0\Gamma c_0\cdots c_{r-1}\cap H_{I_r}
\end{align*}
and we write $\Gamma_r=\pi_r(\tilde{\Gamma}_r)$, where $\pi_r:\gH_{I_r}\rightarrow\gM_{I_r}$ denotes the natural projection. We let $F_r$ denote a $\Gamma_r$--set for $\gM_{I_r}$. The set $I_r=\Delta\setminus\{\alpha_{i_1},\ldots,\alpha_{i_r}\}$ restricts to a set of simple $\QQ$--roots for $\gM_{I_r}$ with respect to $\gA^{I_r}$. We let $\delta_r=\delta(\gH_{I_r},K_{I_r},I_r,F_r)$ and apply Theorem \ref{criterion1} to the $\gH^{(r)}_n\subset\gH_{I_r}$ and $h^{(r)}_n\in H_{I_r}$. We assume that
\begin{align}\label{equal0}
\liminf_{n\rightarrow\infty}\ \delta_r(\gH^{(r)}_n,h^{(r)}_n)=0.
\end{align}

By Proposition \ref{mc} (i), after possibly extracting a subsequence, there exists a maximal couple $(\alpha_{i_{r+1}},(\lambda^{(r+1)}_n)_{n\in\NN})$ for $(\gH^{(r)}_n)_{n\in\mathbb{N}}$ and $(h^{(r)}_n)_{n\in\mathbb{N}}$ with respect to $\gH_{I_r}$, $K_{I_r}$, $I_r$, and $F_r$. Therefore,  if we write $I_{r+1}=I_r\setminus\{\alpha_{i_{r+1}}\}=\Delta\setminus\{\alpha_{i_1},\ldots,\alpha_{i_{r+1}}\}$, then
\begin{align}\label{Hn}
\gH^{(r)}_n\subset\gN_{I_{r}}\lambda^{(r+1)}_n\gP^{I_r}_{I_{r+1}}(\lambda^{(r+1)}_n)^{-1}
\end{align}
and, by Proposition \ref{mc} (ii), if we write
\begin{align*}
(\lambda^{(r+1)}_n)^{-1}h^{(r)}_n=h^{(r+1)}_na^{(r+1)}_nk^{(r+1)}_n\in H_{I_{r+1}}A^{I_{r}}_{I_{r+1}}K_{I_{r}},
\end{align*}
then $\alpha_{i_{r+1}}(a^{(r+1)})\rightarrow\infty$ as $n\rightarrow\infty$. We write $\lambda^{(r+1)}_n=\gamma^{(r+1)}_nc_r\in\Gamma_rF_r$. After possibly enlarging $F_r$ and extracting a subsequence, we may assume that 
\begin{align*}
\gamma^{(r+1)}_n\in c_{r-1}^{-1}\cdots c_0^{-1}\Gamma c_0\cdots c_{r-1}\cap M_{I_r}.
\end{align*}

We iterate this step until we can no longer achieve (\ref{equal0}). We let $r\geq 1$ be the maximal integer such that the previous step goes through and then we proceed as follows.

\subsection*{Step $r+2$}

We have groups
\begin{align*}
\gH^{(r+1)}_n=(\lambda^{(1)}_n\lambda^{(2)}_n\cdots\lambda^{(r+1)}_n)^{-1}\gH_n\lambda^{(1)}_n\lambda^{(2)}_n\cdots\lambda^{(r+1)}_n\subset\gH_{I_{r+1}}
\end{align*}
and elements
\begin{align*}
(\lambda^{(1)}_n\lambda^{(2)}_n\cdots\lambda^{(r+1)}_n)^{-1}g_n&=(\lambda^{(r+1)}_n)^{-1}\cdots(\lambda^{(2)}_n)^{-1}(\lambda^{(1)}_n)^{-1}g_n\\
&=h^{(r+1)}_na^{(r+1)}_n\cdots a^{(1)}_n k^{(r+1)}_n\cdots k^{(1)}_n,
\end{align*}
as before.  We let 
\begin{align*}
\tilde{\Gamma}_{r+1}=c^{-1}_{r}\cdots c^{-1}_0\Gamma c_0\cdots c_{r}\cap H_{I_{r+1}}
\end{align*}
and we write $\Gamma_{r+1}=\pi_{r+1}(\tilde{\Gamma}_{r+1})$, where $\pi_{r+1}:\gH_{I_{r+1}}\rightarrow\gM_{I_{r+1}}$ denotes the natural projection. We let $F_{r+1}$ denote a $\Gamma_{r+1}$--set for $\gM_{I_{r+1}}$.

We have a set $I_{r+1}=\Delta\setminus\{\alpha_{i_1},\ldots,\alpha_{i_{r+1}}\}$, which restricts to a set of simple $\QQ$--roots for $\gM_{I_{r+1}}$ with respect to $\gA^{I_{r+1}}$. We let $\delta_{r+1}=\delta(\gH_{I_{r+1}},K_{I_{r+1}},I_{r+1},F_{r+1})$ and we apply Theorem \ref{criterion1} to the $\gH^{(r+1)}_n\subset\gH_{I_{r+1}}$ and $h^{(r+1)}_n\in H_{I_{r+1}}$ to necessarily find
\begin{align*}
\liminf_{n\rightarrow\infty}\ \delta_{r+1}(\gH^{(r+1)}_n,h^{(r+1)}_n)>0
\end{align*}
(otherwise $r$ was not maximal).
We conclude that the set of homogeneous measures on $\tilde{\Gamma}_{r+1}\backslash H_{I_{r+1}}$ associated with $\gH^{(r+1)}_n$ and $h^{(r+1)}_n$ is sequentially compact. 

Therefore, Theorem \ref{DGUtype} follows (with $\gP=\gP_{I_{r+1}}$, $c=c_0\cdots c_r$, and $\gamma_n=\gamma^{(1)}_n\cdots\gamma^{(r+1)}_n$) from the following proposition.

\begin{prop}\label{rootstoinfty}
For $j=1,\ldots,r+1$, we have 
\begin{align*}
\alpha_{i_j}(a^{(r+1)}_n\cdots a^{(1)}_n)\rightarrow\infty,\text{ as }n\rightarrow\infty.
\end{align*}
\end{prop}

\begin{proof}
We will prove the proposition by induction. The base case is the following.

\begin{lem}
We have
\begin{align*}
\alpha_{i_{r+1}}(a^{(r+1)}_n\cdots a^{(1)}_n)\rightarrow\infty\text{ as }n\rightarrow\infty.
\end{align*}
\end{lem}

\begin{proof}
For $j\geq 1$, we have 
\begin{align*}
a^{(r+1-j)}_n\in A^{I_{r-j}}_{I_{r+1-j}}\subset A_{I_{r+1-j}}\subset A_{I_r}\subset\ker\alpha_{i_{r+1}},
\end{align*}
where the outer inclusions are part of the definitions, and the middle inclusion is justified by Lemma \ref{inc}. We conclude that
\begin{align*}
\alpha_{i_{r+1}}(a^{(r+1)}_n\cdots a^{(1)}_n)=\alpha_{i_{r+1}}(a^{(r+1)}_n)
\end{align*}
and so the result follows from the fact that $\alpha_{i_{r+1}}(a^{(r+1)}_n)\rightarrow\infty$, as $n\rightarrow\infty$ (see Step $r+1$).
\end{proof}

The inductive step is given by the following lemma.

\begin{lem}\label{induction}
Let $0\leq l\leq r-1$ and assume that
\begin{align*}
\alpha_{i_{r+1-j}}(a^{(r+1)}_n\cdots a^{(1)}_n)\rightarrow\infty,\text{ as }n\rightarrow\infty
\end{align*}
for $j=0,\ldots,l$. Then
\begin{align*}
\alpha_{i_{r-l}}(a^{(r+1)}_n\cdots a^{(1)}_n)\rightarrow\infty,\text{ as }n\rightarrow\infty.
\end{align*}
\end{lem}

\begin{proof}
For $s\geq l+2$, we have
\begin{align*}
a^{(r+1-s)}_n\in A^{I_{r-s}}_{I_{r+1-s}}\subset A_{I_{r+1-s}}\subset A_{I_{r-l-1}}\subset\ker\alpha_{i_{r-l}},
\end{align*}
where the outer inclusions are part of the definitions, and the middle inclusion is justified by Lemma \ref{inc}. We conclude that
\begin{align*}
\alpha_{i_{r-l}}(a^{(r+1)}_n\cdots a^{(1)}_n)=\alpha_{i_{r-l}}(a^{(r+1)}_n \cdots a^{(r-l)}_n).
\end{align*}
and so it suffices to show that
\begin{align*}
\alpha_{i_{r-l}}(a^{(r+1)}_n \cdots a^{(r-l)}_n)\rightarrow\infty,\text{ as }n\rightarrow\infty.
\end{align*}

Now we are in the situation of Section \ref{keyresult}, where the ambient group is $\gM_{I_{r-l-1}}$, the maximal split torus is $\gA^{I_{r-l-1}}$, and the set of simple roots is $I_{r-l-1}$. Note that
\begin{align*}
\theta_n=a^{(r+1)}_n \cdots a^{(r-l)}_n=a^{(r+1)}_n\cdots a^{(r-l+1)}_n\cdot a^{(r-l)}_n\in A^{I_{r-l}}_{I_{r+1}}\cdot A^{I_{r-l-1}}_{I_{r-l}}=A^{I_{r-l-1}}_{I_{r+1}}.
\end{align*}
Therefore, if $\alpha_{i_{r-l}}$ is not connected to 
\begin{align*}
I_{r-l-1}\setminus(I_{r+1}\cup\{\alpha_{i_{r-l}}\})=\{\alpha_{i_{r-l+1}},\ldots,\alpha_{i_{r+1}}\}
\end{align*}
in the Dynkin diagram of $\gM_{I_{r-l-1}}$, then we can apply Lemma \ref{discon} (with $I=I_{r+1}$, $\alpha=\alpha_{i_{r-l}}$, and $J=I_{r-l}$) to conclude that
\begin{align*}
\alpha_{i_{r-l}}(a^{(r+1)}_n \cdots a^{(r-l)}_n)=\alpha_{i_{r-l}}(a^{(r-l)}_n).
\end{align*}
(Note that $(I_{r+1}\cup\{\alpha_{i_{r-l}}\})\cap I_{r-1}=I_{r+1}$.) Then the result follows from the fact that $\alpha_{i_{r-l}}(a^{(r-l)}_n)\rightarrow\infty,\text{ as }n\rightarrow\infty$ (by Step $r-l$).

Therefore, suppose that $\alpha_{i_{r-l}}$ is connected to 
\begin{align*}
I_{r-l-1}\setminus(I_{r+1}\cup\{\alpha_{i_{r-l}}\})=\{\alpha_{i_{r-l+1}},\ldots,\alpha_{i_{r+1}}\}
\end{align*}
in the Dynkin diagram of $\gM_{I_{r-l-1}}$. Therefore, by Theorem \ref{rootconj} (with $I=I_{r+1}$, $\alpha=\alpha_{i_{r-l}}$, and $a_n=\theta_n$), Lemma \ref{induction} follows from the following lemma. Note that 
\begin{align*}
\bar{w}^{I_{r-l-1}}_{i_{r-l}}(\theta_n)=\bar{w}^{I_{r-l-1}}_{i_{r-l}}(a^{(r-l)}_n)=\alpha_{i_{r-l}}(a^{(r-l)}_n)^{n_{i_{r-l}}}\rightarrow\infty,\text{ as }n\rightarrow\infty,
\end{align*}
by Step $r-l$, where $n_{i_{r-l}}>0$. Hence, $\bar{w}^{I_{r-l-1}}_{i_{r-l}}(\theta_n)\geq 1$ for $n$ large enough.

\begin{lem}
For $k=r-l+1,\ldots,r+1$, we have
\begin{align*}
\bar{w}^{I_{r-l-1}}_{i_{r-l}}(\theta_n)\geq\bar{w}^{I_{r-l-1}}_{i_{k}}(\theta_n).
\end{align*}
\end{lem}

\begin{proof}
Fix a $k$ as in the statement of the lemma. 

From (\ref{Hn}), we see that $\gH^{(r-l-1)}_n=\lambda^{(r-l)}_n\cdots\lambda^{(r)}_n\gH^{(r)}_n(\lambda^{(r)}_n)^{-1}\cdots(\lambda^{(r-l)}_n)^{-1}$ is contained in
\begin{align*}
\lambda^{(r-l)}_n\cdots\lambda^{(r+1)}_n\gN_{I_r}\gP^{I_r}_{I_{r+1}}(\lambda^{(r+1)}_n)^{-1}\cdots(\lambda^{(r-l)}_n)^{-1}.
\end{align*}
Applying Lemma \ref{para} (with $I=I_r$, $J=I_{r+1}$, $I'=I_{r-l-1}$, and $J'=I_{r-l-1}\setminus\{\alpha_{i_{k}}\}$), we then see that $\gH^{(r-l-1)}_n$ is contained in
\begin{align*}
\lambda^{(r-l)}_n\cdots\lambda^{(r+1)}_n\gN_{I_{r-l-1}}\gP^{I_{r-l-1}}_{I_{r-l-1}\setminus\{\alpha_{i_{k}}\}}(\lambda^{(r+1)}_n)^{-1}\cdots(\lambda^{(r-l)}_n)^{-1}.
\end{align*}
In other words, 
\begin{align}\label{Hr-l-1cont}
\gH^{(r-l-1)}_n\subset\gN_{I_{r-l-1}}\lambda^{(r-l)}_n\cdots\lambda^{(r+1)}_n\gP^{I_{r-l-1}}_{I_{r-l-1}\setminus\{\alpha_{i_{k}}\}}(\lambda^{(r+1)}_n)^{-1}\cdots(\lambda^{(r-l)}_n)^{-1},
\end{align}
where we use the fact that $\lambda^{(r-l)}_n\cdots\lambda^{(r+1)}_n\in\gM_{I_{r-l-1}}(\QQ)$ which normalizes $\gN_{I_{r-l-1}}$.

By definition, we have
\begin{align}\label{hr-2}
(\lambda^{(r+1)}_n)^{-1}\cdots(\lambda^{(r-l)}_n)^{-1}h^{(r-l-1)}_n
=h^{(r+1)}_na^{(r+1)}_n\cdots a^{(r-l)}_nk^{(r+1)}_n\cdots k^{(r-l)}_n.
\end{align}
By Lemma \ref{tori} (with $I_1=I_{r-l-1}$, $I_3=I_{r+1}$, and $I_2=I_{r-l-1}\setminus\{\alpha_{i_k}\}$), we have a direct product decomposition
\begin{align*}
A^{I_{r-l-1}}_{I_{r+1}}=A^{ I_{r-l-1}\setminus\{\alpha_{i_{k}}\}}_{I_{r+1}}A^{I_{r-l-1}}_{I_{r-l-1}\setminus\{\alpha_{i_{k}}\}},
\end{align*}
and so we can write
\begin{align*}
a^{(r+1)}_n\cdots a^{(r-l+1)}_n\cdot a^{(r-l)}_n\in A^{I_{r-l}}_{I_{r+1}}\cdot A^{I_{r-l-1}}_{I_{r-l}}
\end{align*}
as
\begin{align*}
b^{(r-l-1,k)}_n\cdot c^{(r-l-1,k)}_n\in A^{ I_{r-l-1}\setminus\{\alpha_{i_{k}}\}}_{I_{r+1}}\cdot A^{I_{r-l-1}}_{I_{r-2}\setminus\{\alpha_{i_{k}}\}}.
\end{align*}
Therefore, from (\ref{hr-2}), we obtain
\begin{align}\label{langlands}
(\lambda^{(r+1)}_n)^{-1}&\cdots(\lambda^{(r-l)}_n)^{-1}h^{(r-l-1)}_n=h^{(r+1)}_nb^{(r-l-1,k)}_n\cdot c^{(r-l-1,k)}_n\cdot k^{(r+1)}_n\cdots k^{(r-l)}_n,
\end{align}
which is the rational Langlands decomposition in $H_{I_{r-l-1}\setminus\{\alpha_{i_{k}}\}}\cdot A^{I_{r-l-1}}_{I_{r-l-1}\setminus\{\alpha_{i_{k}}\}}\cdot K_{I_{r-l-1}}$, where we use the facts that 
\begin{align*}
h^{(r+1)}_n\in H_{I_{r+1}}\subset H_{I_{r-l-1}\setminus\{\alpha_{i_k}\}}\hspace{.5cm}\text{and}\hspace{.5cm}b^{(r-l-1,k)}_n\in A^{ I_{r-l-1}\setminus\{\alpha_{i_k}\}}_{I_{r+1}}\subset H_{I_{r-l-1}\setminus\{\alpha_{i_k}\}}.
\end{align*}	 

Therefore, since $(\alpha_{i_{r-l}},(\lambda^{(r-l)}_n)_{n\in\NN})$ was a maximal couple for $(\gH^{(r-l-1)}_n)_{n\in\mathbb{N}}$ and $(h^{(r-l-1)}_n)_{n\in\mathbb{N}}$ with respect to $\gH_{I_{r-l-1}}$, $K_{I_{r-l-1}}$, $I_{r-l-1}$, and $F_{r-l-1}$, it follows from (\ref{Hr-l-1cont}) and (\ref{langlands}) that
\begin{align*}
\bar{w}^{I_{r-l-1}}_{i_{r-l}}(\theta_n)=\bar{w}^{I_{r-l-1}}_{i_{r-l}}(a^{(r+1)}_n\cdots a^{(r-l)}_n)&=\bar{w}^{I_{r-l-1}}_{i_{r-l}}(a^{(r-l)}_n)\\
&\geq \bar{w}^{I_{r-l-1}}_{i_{k}}(c^{(r-l-1,k)}_n)\\
&=\bar{w}^{I_{r-l-1}}_{i_{k}}(b^{(r-l-1,k)}_nc^{(r-l-1,k)}_n)\\
&=\bar{w}^{I_{r-l-1}}_{i_{k}}(a^{(r+1)}_n\cdots a^{(r-l)}_n)=\bar{w}^{I_{r-l-1}}_{i_{k}}(\theta_n),
\end{align*}
where the second and third equalities are consequences of Lemma \ref{trivial}. 
\end{proof}
This completes the proof of Lemma \ref{induction}.
\end{proof}
This completes the proof of Proposition \ref{rootstoinfty}.
\end{proof}
This completes the proof of Theorem \ref{DGUtype}.

\section{Appendix: An inequality for simple roots and dual weights}\label{appen}

\begin{center}
\textsc{Jialun Li} 
\end{center}

\vspace{0.5cm}

The purpose of this appendix, is to prove Theorem \ref{rootconj}. The main ingredient is an inequality between simple roots and dual weights. For this discussion, we refer without further mention to \cite{bourbaki46}.

Let $E$ be a linear space. Let $\Pi$ be a root system in $E$, which generates $E$, and let $(\cdot ,\cdot )$ be the inner product on $E$ invariant under the Weyl group. Fix a set $\Delta$ of simple roots in $\Pi$. Let $\{w_\alpha \}_{\alpha\in\Delta}$ be the set of dual weights in $E$, which are defined by the relations
\[( w_\alpha,\beta)=\delta_{\alpha\beta} \]
for $\beta\in\Delta$, where $\delta_{\alpha\beta}$ is the Kronecker symbol. The set of dual weights and the set of simple roots $\Delta$ form two bases of $E$. Using the inner product, we can easily compute the coefficients in the transition matrix. Then we have a relation between simple roots and dual weights,
\begin{equation}\label{alpha}
	\alpha=( \alpha,\alpha) w_\alpha+\sum_{\beta\in\Delta\backslash\{\alpha\}}( \alpha,\beta) w_\beta\quad\hbox{for  $\alpha\in\Delta$}
\end{equation}
and
\begin{equation}\label{eq:omega}
	w_\alpha=\sum_{\beta\in\Delta}( w_\alpha,w_\beta )\beta\quad\hbox{for  $\alpha\in\Delta$.}
\end{equation}
For $\alpha\in\Delta$, set 
$$
d_\alpha=\sum_{\beta\in\Delta}( w_\alpha,w_\beta).
$$
We recall that $( w_\alpha,w_\beta)\ge 0$, and
$( w_\alpha,w_\beta)>0$ if the root system is irreducible. In particular, it follows that $d_\alpha>0$.
We define the weighted dual weights as
\begin{equation}\label{baromega}
\bar{w}_\alpha=w_\alpha/d_\alpha=\frac{\sum_{\beta\in\Delta}( w_\alpha,w_\beta )\beta}{\sum_{\beta\in\Delta}( w_\alpha,w_\beta)}.
\end{equation}
We denote by $E^*$ the dual space of $E$ and identify $E$ with $(E^*)^*$.
For $I\subset \Delta$, we define
$$
\mathfrak{a}_I=\bigcap_{\beta\in I}\ker\beta\subset E^*.
$$
 
 \begin{teo}\label{th:roots}
 	Let $\alpha\in\Delta$ and let $I\subset\Delta\backslash\{\alpha\}$. For every $a\in \mathfrak{a}_I$ satisfying
 	 \begin{align*}
 	 \bar{w}_\alpha(a)\geq\bar{w}_\gamma(a)\;\; \hbox{for all $\gamma\in\Delta\backslash I$}\quad\quad\hbox{and}\quad\quad
 	 \bar{w}_\alpha(a)\ge 0,
 	 \end{align*}
 	 the estimate
 	 $\alpha(a)\geq \bar{w}_\alpha(a)$
 	 also holds.
 \end{teo}

From this, we deduce the following corollary, which is used in the proof of Theorem \ref{tI1}.

\begin{cor}\label{cor:roots}
	Let $\alpha\in\Delta$ and let $I\subset\Delta\backslash\{\alpha\}$ such that $\alpha$ is connected to 
	$\Delta\backslash(I\cup\{\alpha\})$ in the Dynkin diagram. For each $n\in\mathbb{N}$, let 
	$a_n\in \mathfrak{a}_I$ and suppose that 
	$$
	\beta(a_n)\to \infty,\quad\hbox{as $n\to\infty$}
	$$
	for all $\beta\in \Delta\backslash (I\cup\{\alpha\})$.
	Furthermore, suppose that
	$$
	\bar w_\alpha(a_n)\ge \bar w_\beta(a_n)\;\;\hbox{for all $\beta\in \Delta\backslash I$}\quad\quad\hbox{and}\quad\quad \bar w_\alpha(a_n)\ge 0.
	$$ 
	Then
	$$
	\alpha(a_n)\to \infty,\quad\hbox{as $n\to\infty$.}
	$$ 
\end{cor}

\begin{proof}[Proof of Corollary \ref{cor:roots}]
Using Theorem \ref{th:roots} and 	\eqref{eq:omega}, we obtain 
 \[\alpha(a_n)\geq \bar{w}_\alpha(a_n)=\frac{1}{d_\alpha}
 \sum_{\beta\in\Delta\backslash I}( w_\alpha,w_\beta )\beta(a_n), \]
 and 
\[\left(1-\frac{( w_\alpha,w_\alpha )}{d_\alpha}\right)\alpha(a_n)\geq 
\frac{1}{d_\alpha}\sum_{\beta\in\Delta\backslash (I\cup\{\alpha\})}( w_\alpha,w_\beta )\beta(a_n). \]
Here $( w_\alpha,w_\beta )\ge 0$ and, moreover, it follows
from our connectedness assumption that $( w_\alpha,w_\beta )> 0$ for at least
one $\beta\in \Delta\backslash (I\cup\{\alpha\})$. In particular,
$( w_\alpha,w_\alpha )<d_\alpha$.
Hence, the last estimate implies the corollary.
\end{proof}

Now we proceed with the proof of Theorem \ref{th:roots}.
We first prove the case  when $I=\emptyset$.

\begin{proof}[Proof of Theorem \ref{th:roots} for $I=\emptyset$]
	
	By \eqref{alpha} and \eqref{baromega}, we know
	\begin{equation}\label{alpha2}
			\alpha=( \alpha,\alpha) d_\alpha\bar{w}_\alpha+\sum_{\beta\neq\alpha}(\alpha,\beta ) d_\beta\bar{w}_\beta.
	\end{equation}
	The set of simple roots forms a basis of $E$. By \eqref{baromega}, the term $\bar{w}_\beta$ is a linear combination of simple roots, and the sum of the coefficients in the expression equals $1$. Therefore by computing the coefficients of simple roots, \eqref{alpha2} implies that
	\begin{equation}\label{2d}
	( \alpha,\alpha) d_\alpha+\sum_{\beta\neq\alpha}( \alpha,\beta) d_\beta=1.
	\end{equation}
	Due to properties of simple roots, we know that $(\alpha,\beta)\leq 0$ for $\alpha\neq\beta$. Hence, \eqref{alpha2}, \eqref{2d}, and the hypothesis imply
	\[\alpha(a)=\bar{w}_\alpha(a)+\sum_{\beta\neq\alpha}(-(\alpha,\beta )) d_\beta(\bar{w}_\alpha(a)-\bar{w}_\beta(a))\geq \bar{w}_\alpha(a).\]
	The proof is complete.
%	Since the root system is irreducible, we have $d_\alpha=\sum_{\beta\in\Delta}(w_\alpha,w_\beta )>(w_\alpha,w_\alpha )$, which implies that the coefficient of $\alpha$ in $\bar{w}_\alpha$ is less than $1$. 
\end{proof}
To prove the general case of the theorem, we need the following lemma.
\begin{lem}\label{positive}
	For $I\subset \Delta$, the set
	$I\cup\big\{w_\gamma:\,\gamma\in \Delta\backslash I \big\}$ forms a basis of $E$. Moreover, for every $\alpha\in\Delta$, we have
\[\alpha=\sum_{\beta\in I} c_\beta\beta+\sum_{\gamma\in\Delta\backslash I}c_\gamma w_\gamma, \]
where $c_\delta\leq 0$ for $\delta\in \Delta\backslash\{\alpha\}$.
\end{lem}
\begin{proof}[Proof for the general case of Theorem \ref{th:roots}]
 	Using the expression in the Lemma \ref{positive} and \eqref{eq:omega}, by the same argument as for the special case $I=\emptyset$, we see that
	\[1=\sum_{\beta\in I}c_\beta+\sum_{\gamma\in\Delta\backslash I}c_\gamma d_\gamma\leq \sum_{\gamma\in\Delta\backslash I}c_\gamma d_\gamma, \]
	where the last inequality is due to Lemma \ref{positive}. 
	Therefore, by the definition of $\mathfrak{a}_I$ and Lemma \ref{positive},
	for any $a\in \mathfrak{a}_I$,
	\begin{align*}
	\alpha(a)&=\sum_{\gamma\in\Delta\backslash I}c_\gamma d_\gamma\bar{w}_\gamma(a).
	\end{align*}
	Hence, using that $c_\gamma\le 0$ for $\gamma\ne \alpha$, we deduce that 
	\begin{align*}
%	&=\left(\sum_{\gamma\in\Delta\backslash I}c_\gamma d_\gamma\right) \bar{w}_\alpha(a)+\sum_{\gamma\in\Delta\backslash (I\cup\{\alpha\})}(-c_\gamma d_\gamma)(\bar{w}_\alpha(a)-\bar{w}_\gamma(a))\\
	\alpha(a)&\geq \left(\sum_{\gamma\in\Delta\backslash I}c_\gamma d_\gamma\right)\bar{w}_\alpha(a)\ge \bar{w}_\alpha(a),
	\end{align*}
	since $\bar{w}_\alpha(a)\geq 0$.
	The proof is complete.
\end{proof}
It remains to prove Lemma \ref{positive}.
We first recall two facts.
\begin{lem}\label{dynkin}
	A connected subgraph of a Dynkin diagram is a Dynkin diagram.
\end{lem}
\begin{lem}\label{cartan}
	The inverse of the Gramm matrix $\big((\alpha,\beta) \big)_{\alpha,\beta\in\Delta}$ of an irreducible root system $\Delta$ is a matrix with positive entries.
\end{lem}

Lemma \ref{cartan} is usually formulated in terms of the Cartan matrix
of the root system that has entries $\frac{2(\alpha,\beta)}{(\beta,\beta)}$ with $\alpha,\beta\in\Delta$.
Since the Cartan matrix is the product of the Gramm matrix and 
the diagonal matrix with the positive entries 
$\frac{2}{(\beta,\beta)}$ with $\beta\in\Delta$, the claim is also true 
for the Gramm matrix as well.

\begin{proof}[Proof of Lemma \ref{positive}]
We index the simple roots $\Delta=\{\alpha_1,\cdots,\alpha_n \}$ so that $I=\{\alpha_1,\cdots,\alpha_m \}$ and so that 
$$
\{\alpha_1,\cdots,\alpha_{k_1}\},\{\alpha_{k_1+1},\cdots,\alpha_{k_1+k_2}\},\ldots , \{\alpha_{k_1+\cdots +k_{l-1}+1},\cdots,\alpha_{k_1+\cdots+k_l}=\alpha_m\}
$$ are nonadjacent connected subgraphs in the Dynkin diagram of $\Delta$.
We observe that
$$
{}^t(\alpha_1,\cdots,\alpha_n)=A \cdot {}^t(w_{\alpha_1},\cdots,w_{\alpha_n})
$$
where 
$A=\begin{pmatrix}( \alpha_i,\alpha_j)
\end{pmatrix}_{1\leq i,j\leq n},$
and 
$$
{}^t (\alpha_1,\cdots,\alpha_m,w_{\alpha_{m+1}},\cdots,w_{\alpha_n} )
= \begin{pmatrix}B & C\\
0& {\rm Id}_{n-m}
\end{pmatrix} \cdot {}^t(w_{\alpha_1},\cdots,w_{\alpha_n}),
$$
where $B=\begin{pmatrix}( \alpha_i,\alpha_j)
\end{pmatrix}_{1\leq i,j\leq m}$ and $C=\begin{pmatrix}( \alpha_i,\alpha_j)
\end{pmatrix}_{1\leq i\leq m, m+1\leq j\leq n}.$ 
We note that $B$ is invertible by Lemma \ref{dynkin}.
Therefore, 
$$
{}^t(\alpha_1,\cdots,\alpha_n)=D\cdot {}^t  (\alpha_1,\cdots,\alpha_m,w_{\alpha_{m+1}},\cdots,w_{\alpha_n} ),
$$
where
\[
D=A \begin{pmatrix}B & C\\
0& {\rm Id}_{n-m}
\end{pmatrix}^{-1}=A \begin{pmatrix}B^{-1} & -B^{-1}C\\
0& {\rm Id}_{n-m}
\end{pmatrix}.\]
By our assumption, the matrix $B$ is block-diagonal consisting of $l$ blocks. By Lemma \ref{dynkin}, each block is the Gramm matrix of a Dynkin diagram. 
%(In face the Cartan matrix equals $\begin{pmatrix}
%n(\alpha_i,\alpha_j)\end{pmatrix}_{i,j}=\begin{pmatrix}2\frac{(\alpha_i,\alpha_j )}{( \alpha_j,\alpha_j)}\end{pmatrix}_{i,j}=\begin{pmatrix}(\alpha_i,\alpha_j )\end{pmatrix}_{i,j}diag\{
%\frac{2}{( \alpha_j,\alpha_j)}\}$). 
In particular, it follows that $B$ is invertible and, by Lemma \ref{cartan},
the inverse of $B$ has non-negative entries. 
Since $( \alpha_i,\alpha_j)\leq 0$ for all $i\neq j$, the entries of the matrix $C$ are non-positive. Hence, the matrix $-B^{-1}C$  also has non-negative entries.

Now we can compute the coefficients in Lemma \ref{positive}. 
When $\alpha\in I$, $c_\beta=\delta_{\alpha \beta}$ and, in particular,
$c_\beta=0$ for $\alpha\ne \beta$. We suppose that
$\alpha\notin I$, so that  $\alpha=\alpha_p$ with some $p>m$. Then the coefficients in the expression of $\alpha_p$ with respect to $(\alpha_1,\cdots,\alpha_m,w_{\alpha_{m+1}},\cdots,w_{\alpha_n} )$ are given by
\[\begin{pmatrix}( \alpha_p,\alpha_j)
\end{pmatrix}_{1\leq j\leq n}\begin{pmatrix}B^{-1} & -B^{-1}C\\
0& {\rm Id}_{n-m}
\end{pmatrix}. \]
Here the matrices $B^{-1}$ and $-B^{-1}C$ have non-negative entries
and $(\alpha_p,\alpha_j )\le 0$ for all $j\ne p$.
Hence, performing matrix multiplication, we deduce that all 
the coefficients except the $p^{\rm th}$ one are non-positive.
\end{proof}

\bibliography{EquiDraft}
\bibliographystyle{alpha}

\end{document}